\documentclass[10pt,a4paper]{article}
\usepackage[utf8]{inputenc}
\usepackage[T1]{fontenc}
\usepackage{amsthm}
\usepackage{amsmath}
\usepackage{amssymb}
\usepackage{makeidx}
\usepackage{graphicx}
\usepackage[width=16.00cm, height=22.00cm]{geometry}
\usepackage[english]{babel}
\usepackage{caption}
\usepackage[dvipsnames]{xcolor}
\numberwithin{equation}{section}
\usepackage[toc,page]{appendix}
\usepackage{wrapfig}

\newtheorem{theorem}{Theorem}[section]

\newtheorem{proposition}[theorem]{Proposition}

\newtheorem{remark}{Remark}
\newtheorem{definition}{Definition}
\newtheorem{assumption}{Assumption}

\numberwithin{equation}{section}

\newcommand{\xx}{{\textbf x}}

\newcommand{\ee}{{\mathrm e}}
\newcommand{\dd}{{\mathrm d}}

\newcommand{\ii}{{\mathrm i}}
\newcommand{\error}{\mathcal{E}_\omega^h}
\definecolor{light-gray}{gray}{0.98}

\DeclareMathOperator{\R}{\mathbb{R}}

\title{Effective  highly accurate time integrators for linear Klein--Gordon equations across the scales}

\begin{document}
	
\author{Karolina Kropielnicka\footnote{Institute of Mathematics, Polish Academy of Sciences, Warsaw, Poland} \and Karolina Lademann\footnote{Institute of Mathematics, Physics and Computer Science, University of Gda\'{n}sk, Gda\'{n}sk, Poland} \and Katharina Schratz \footnote{Laboratoire Jacques-Louis Lions, Sorbonne Universit\'e, Paris, France}}

		\maketitle

	\abstract{We propose an efficient approach for  time integration of  Klein--Gordon equations with  highly oscillatory in time input terms. The new methods are highly accurate in the entire range, from slowly varying up to highly oscillatory regimes. Our approach is based on splitting methods tailored to the structure of the input term which allows us to  resolve the oscillations in the system uniformly in all frequencies, while the error constant does not grow as the oscillations increase. Numerical experiments highlight our theoretical findings and  demonstrate  the efficiency of the new schemes.}


\section{Introduction}

High oscillation is a phenomenon occurring in a very wide range of problems in science and engineering and it represents an enduring computational challenge. The main reason is that classical numerical analysis rests upon the concept of truncated Taylor expansion with a small error constant, but this need not be true once high oscillation is present: think for example about the amplitude of the derivatives of $\cos\omega x$ for $\omega\gg1$. Highly oscillatory phenomena, not least highly oscillatory partial differential equations, have been a subject of very active research in the last three decades \cite{hmm,hi-ocs-int,hop}. As things stand, there exists a fairly robust theory for phenomena which are driven by high frequency, as well as of course problems that exhibit no high oscillation. However, there exists a dearth of computational tools for problems, often to be found in applications, where oscillation occurs across a wide range of scales, from non-oscillatory to highly oscillatory. This paper is concerned with problems of this kind. While our long-term objective is to fashion tools for general time-dependent partial differential equations of this type, the paper focusses on the specific example of linear Klein--Gordon equation with highly oscillatory term. We bring a wide range of tools to bear on this problem, not least the Magnus expansion and exponential splittings, while accompanying our presentation with careful error estimates and numerical examples. 

We consider the  equation
\begin{eqnarray}\label{KG1}
	\partial^2_t   \psi(\xx ,t) & =& \Delta  \psi(\xx ,t) + f (\xx,t) \psi(\xx ,t) 
\end{eqnarray}
with  initial conditions $\psi(\xx,0) = \psi_0(\xx)$, $\partial_t \psi (\xx,t_0) = \varphi_0(\xx)$ and  periodic boundary conditions\footnote{The assumption of periodic boundary conditions is solely for the sake of simplicity of exposition and all proposed methods can be easily generalised to other boundary conditions.}, that is $\xx\in \mathbb{T}^d$. Here, $f(\xx,t)$ is a given function of the form
\begin{eqnarray}\label{input_term}
	f(\xx,t)& =&\alpha(\xx,t)+\sum_{|n|\leq N}a_n(\xx,t)\ee^{\ii \omega_n t},
\end{eqnarray}
where $n\in\mathbb{N}$, $\omega_n\in\mathbb{R}$, $\omega_{\min}=\min_{n\leq N}|\omega_n|\geq 1$ and $\omega_{\max}=\max_{n\leq N} |\omega_n|<\infty$. Functions $\alpha(\xx,t)  $ and  $a_n(\xx,t)$, $n \leq N$, do not oscillate in the $\omega_n$s (i.e. their time derivatives are bounded independently of the $\omega_n$s). 
It is enough to assume smoothness of the functions $\alpha$ and $a_n$. For the precise assumptions on system \eqref{KG1} we refer to Section \ref{ESTIMATES}. 

A special case of (\ref{KG1}) is the Klein--Gordon equation in atomic scaling (that is where the Planck constant $\hbar$ and speed of light $c$ equal $1$) with  time and space dependent mass ($ m^2 = m^2(\xx,t) $) which was first proposed by M.\@ Znojil in \cite{zno1} and \cite{zno2} and which finds its applications in quantum cosmology, see \cite{mos1}, \cite{zno3}. The presence of high oscillations in the Klein--Gordon equation can have important implications for the behavior of the scalar field. For example, it can lead to the production of particles in differential equations modeling early universe evolution.

Linear and nonlinear Klein--Gordon equations have recently gained a lot of attention in computational mathematics, see for instance \cite{bao2,12,6,shark,yus}, and  references therein. The highly oscillatory case $\omega_{\rm max} \gg 1$ is thereby particularly challenging as classical methods in general fail to resolve the underlying oscillatory structure  of the solution, leading to large errors and huge computational costs. Up to our knowledge there are only three works devoted to the computational approaches in the case of space- and time-dependent input terms. The earliest result is described in  \cite{baderblanes}, where authors propose  fourth and sixth order splittings based on the commutator-free, Magnus type expansions. The methods perform very well in the non-oscillatory regimes, however, introduce large error constants in case of high frequencies $\omega_{\rm max} \gg 1$. The result obtained in \cite{asympt} proposes a  method specialized to purely highly oscillatory input terms meaning that  $\omega_{\rm min} \gg 1$. The drawback of the method is that it  drastically  fails in the non-oscillatory regimes  where $\omega_{\rm min} \sim 1$. Note that none of these methods deals effectively in case of $\omega_{\rm min} \sim 1$ and $\omega_{\rm max} \gg 1$.  Recently obtained, unpublished yet, result, see \cite{3order},  proposes a third order uniform in low to high oscillatory regimes method based on the Duhamel principle approach. Note that classical methods, like finite differences for example,  for \eqref{KG1} introduce a time error of type $(h  \cdot \omega_{\max})^p$ leading to severe step size restrictions $h < 1/\omega_{\max}$, loss of convergence and huge computational costs in the highly oscillatory cases.\\

In this paper we substantially improve all the aforementioned results. Unlike \cite{baderblanes} and \cite{asympt}, the new methods are effective for all three kinds of regimes: slowly oscillatory ($ \omega_{\rm max} \sim 1$), highly oscillatory ($ \omega_{\rm max}\gg 1$) and ranging across all oscillation scales ($\omega_{\rm min}\sim 1$ and $\omega_{\rm max}\gg 1$) in the input term. This is a challenging goal because, while non-oscillatory terms are traditionally modelled using a Taylor expansion, highly oscillatory terms yield themselves best to asymptotic expansions.

The proposed approach improves upon the method described in \cite{3order}. First of all,  third-order convergence is raised to  fourth-order. The concept of order, however, is of purely formal nature as the desired convergence rate is not affected by large error constants only once $h <1/\omega_{\max}$. What matters in  realistic applications (especially when $h \geq1/\omega_{\max}$) is the accuracy in which the proposed schemes outperform all other existing methods. This can be observed both  theoretically and experimentally.

Theoretical investigation confirms that the proposed methods optimise precision in different regimes of step size and frequency. As claimed in Theorem \ref{thm:errorS} of this paper, the leading term of the error scales like $\mathcal{O}(h^5+\min\{h^3,h^2/\omega_{\min},h^5\omega_{\max}^2\})$ and like $\mathcal{O}(\min\{h^3,h^2/\omega_{\min},h^5\omega_{\max}^2\})$ in the special case $\alpha(\xx,t)\equiv 0$. The most pessimistic case of this estimate equals the most optimistic case obtained in  \cite{3order}.

Section \ref{EXPERIMENTS} abounds with numerous computational simulations and comparisons showing that the accuracy of the method, not its formal order,  plays the crucial role in the case of high oscillatory problems. In particular we refer to  \textit{Example 5\/}, see Fig. \ref{fig:omegas_sum}, where the whole range of oscillations appear in the input term and where the accuracy and efficiency of the proposed method prove to be outperforming all other approaches.

Let us emphasise that in this paper we are interested solely in time integration, since the problematic oscillations occur in time. We are not concerned with the specific method of space discretisation, which can be chosen depending on the boundary conditions,  constraints arising from applications and  the preferences of the user. In our numerical examples we  resort to Fourier spectral collocation in space.

The new approach is based on well known tools like Magnus expansion, Strang splitting and compact splittings. The motivation and  derivation of the methods is presented in Section \ref{DERIVATION}. Section \ref{ESTIMATES} deals with  rigorous estimates of the leading error terms, which were disregarded in the derivation of the schemes. In Section \ref{ERRORS} we provide the structure of the leading error terms committed by the proposed schemes and illustrate it graphically. In Section \ref{EXPERIMENTS} we illustrate our theoretical findings with  numerical experiments. In addition, we compare  various  existing methods with the new approach and highlight the behavior of the latter with respect to high frequencies  $\omega_n$. Appendices \ref{apA} and \ref{strang_est} present calculations which may be of interest while reading Section \ref{ESTIMATES}.

\section{Derivation of the methods}\label{DERIVATION}

We are looking for solutions $\psi(t):=\psi(\,\cdot\,,t) \in C^2([0,T],H^{s+6}(\mathbb{T}^d))$ and assume that $f(t):=f(\,\cdot\,,t) \in C^2([0,T],H^{s+4}(\mathbb{T}^d))$, $s\geq 0$. Let us observe that equation (\ref{KG1}) can be analyzed in its abstract form
\begin{eqnarray}\label{system_for_z}
	\partial_t z(t)=A(t)z(t), \quad 
	{\rm where}\ \quad
	z(t)=\left[  \begin{array}{c}
		\psi(t) \\
		\partial t/ \partial \, \psi(t) 
	\end{array} \right],
\end{eqnarray}
where
\begin{eqnarray}\label{A}
	A(t)=
	\left[  \begin{array}{cc}
		0 & 1 \\
		N(t) & 0 
	\end{array} \right]
	\quad
	{\rm with}\ \quad
	N(t)=\Delta +f(t).
\end{eqnarray}

Analytical solution of (\ref{system_for_z}) can be represented as the well-known Fer, Magnus or Dyson expansions. Like in the case of the Schr\"{o}dinger equation in \cite{iserles2019compact}   we will resort to the two leading terms of Magnus series, to which we will apply Strang splitting, followed by the compact scheme of \cite{chin2002gradient} type to obtain the {\it Ma-St-CC method} of fourth order. Indeed, the first two terms of Magnus expansion are sufficient to obtain fourth order approximations, see \cite{iserles00lgm,INR2001,blanesros}. What is more, the first two terms of Magnus expansion scale like $\mathcal{O}(h)$ and $\mathcal{O}(h^3)$ respectively, so Strang splitting applied to the truncation maintains the fourth order accuracy. As explained in \cite{iserles2019compact} compact splitting of \cite{chin2002gradient} type applied to the inner exponent  of $\mathcal{O}(h)$ order also maintains  fourth order of the entire method. 

Similar order of convergence can be expected in the case of the Klein--Gordon equation. The  highly oscillatory  input term, however, results in huge error constant, spoiling the accuracy of the method. For this reason instead of searching for fourth order convergence  (that is local error scaling like $\mathcal{O}(h^5)$) we  will  look for the accuracy depending on the relation between time step $h$, $\omega_{\min}$ and $\omega_{\max}$, obtaining local error scaling like  $\mathcal{O}(h^5+\min\{h^3,h^2/\omega_{\min},h^5\omega_{\max}^2\})$ and  like $\mathcal{O}(\min\{h^3,h^2/\omega_{\min},h^5\omega_{\max}^2\})$ in the special case $\alpha(\xx,t)\equiv 0$, see Theorem \ref{thm:errorS} of this paper.

In this Section we  present the full derivation of Ma-St-CC method tailored for the Klein--Gordon equation and  indicate the cut-off terms responsible for the growth of the error constant. Estimates of the indicated leading error terms will be presented in Section \ref{ESTIMATES}.

\subsection{Truncation of Magnus expansion}\label{MAGNUS}

The Magnus expansion \cite{magnus1954exponential} applied to  (\ref{system_for_z}) is of the form
\begin{eqnarray}\label{6}
	z(t+h)=\ee^{\Theta(t+h,t)}z(t),
\end{eqnarray}
where $\Theta(t+h,t)=\sum_{k=1}^\infty\Theta_k(t+h,t)$, and each $\Theta_k(t+h,t)$ is a linear combination of  $k$th times nested integrals of $(k-1)$-fold nested commutators,

\begin{align}\label{mag1}
	\Theta_1(t+h,t) &= \int_0^h A(  t+ t_1) \dd   t_1 \\
	\label{mag2}
	\Theta_2(t+h,t) &=- \dfrac{1}{2} \int_0^h \int_0^{  t_1} \left[ A(  t+ t_2) , A(  t+ t_1) \right] \dd   t_2 \dd   t_1  	\\
	\label{mag3}
	\Theta_3(t+h,t) &=\dfrac{1}{6} \int_0^h \int_0^{  t_1} \int_0^{  t_2}  \left[ A(  t+ t_1) , \left[ A(  t+ t_2) , A(  t+ t_3) \right] \right] \dd   t_3  \dd   t_2 \dd   t_1\\
\nonumber
	& + \dfrac{1}{6} \int_0^h \int_0^{  t_1} \int_0^{  t_2}  \left[ A(  t+ t_3) , \left[ A(  t+ t_2) , A(  t+ t_1) \right] \right] \dd   t_3  \dd   t_2 \dd   t_1 \end{align}
	\begin{align}
	\label{mag4}
	\Theta_4(t+h,t) &= \dfrac{1}{12} \int_0^h \int_0^{  t_1} \int_0^{  t_2}  \int_0^{  t_3} \left[ \left[ \left[ A(t+t_1),A(t+t_2) \right] , A(t+t_3)\right] ,A(t+t_4) \right]  \dd   t_4 \dd   t_3  \dd   t_2 \dd   t_1\\
	\nonumber
	& + \dfrac{1}{12} \int_0^h \int_0^{  t_1} \int_0^{  t_2}  \int_0^{  t_3} \left[ A(t+t_1), \left[  \left[  A(t+t_2), A(t+t_3)\right] , A(t+t_4) \right] \right]  \dd   t_4 \dd   t_3  \dd   t_2 \dd   t_1\\
	\nonumber
	& + \dfrac{1}{12} \int_0^h \int_0^{  t_1} \int_0^{  t_2}  \int_0^{  t_3} \left[ A(t+t_1), \left[ A(t+t_2),\left[ A(t+t_3) , A(t+t_4)\right] \right] \right]   \dd   t_4 \dd   t_3  \dd   t_2 \dd   t_1 \\
	\nonumber
	& + \dfrac{1}{12} \int_0^h \int_0^{  t_1} \int_0^{  t_2}  \int_0^{  t_3} \left[ A(t+t_2), \left[ A(t+t_3),\left[ A(t+t_4) , A(t+t_1)\right] \right] \right]   \dd   t_4 \dd   t_3  \dd   t_2 \dd   t_1.
\end{align}	
The Ma-St-CC method is based on the first two terms,
\begin{eqnarray}\label{4th_Magnus}
	z(t+h)\approx\exp\left(\int_0^hA(t+ t_1)\dd  t_1- \frac{1}{2} \int_0^h \int_0^{  t_1} \left[ A(  t+ t_2) , A(  t+ t_1) \right] \dd   t_2 \dd   t_1\right)z(t).
\end{eqnarray}

Expressions $\Theta_3(t+h,t) $ and $\Theta_4(t+h,t) $  constitute  principal error terms of the Ma-St-CC method. As will be explained in the end of Subsection \ref{MAGNUS_CALCULATIONS} the magnitudes of components $\Theta_k(t+h,t) $, $k\geq 5$ are of  $h^{k-4}$ times smaller than $\Theta_4(t+h,t) $. The estimates of  $\Theta_3(t+h,t) $ and $\Theta_4(t+h,t) $, showing that they scale like $\mathcal{O}(h^5+\min\{h^3,h^2/\omega_{\min},h^5\omega_{\max}^2\})$ and  like $\mathcal{O}(\min\{h^3,h^2/\omega_{\min},h^5\omega_{\max}^2\})$ in the special case $\alpha(\xx,t)\equiv 0$, will be presented in Subsection \ref{MAGNUS_CALCULATIONS}.

\subsection{The Strang splitting}\label{STRANG}
The exponent appearing on the right hand side of  (\ref{4th_Magnus}) is computationally costly. Indeed, as will be shown  in Sections \ref{outer_term} and \ref{inner_term} the integrals $\int_0^hA(t+ t_1)\dd  t_1$ and $\int_0^h \int_0^{  t_1} \left[ A(  t+ t_2) , A(  t+ t_1) \right] \dd   t_2 \dd   t_1$ differ not only in their structure (the first one is anti-diagonal, while the second is diagonal), but also in the magnitude of order in which they scale (in terms of $h$ and  $\omega_n$). We will split this exponent using  the Strang decomposition 
\begin{eqnarray}\label{Strang_general}
	\exp(X+Y)\approx\exp\!\left(\frac{1}{2}X\right)\exp(Y)\exp\!\left(\frac{1}{2}X\right),
\end{eqnarray}
whose convergence was proved, for example,  in \cite{jahnke2000error}, and whose accuracy depends on the commutators $[Y,[Y,X]]$ and $[X,[X,Y]]$.

Decomposition (\ref{Strang_general}) applied to (\ref{4th_Magnus})  results in the splitting
\begin{eqnarray}\label{4th_Strang}
	z(t+h)&\approx&\exp\!\left(-\frac{1}{4}\int_0^h \int_0^{  t_1} \left[ A(  t+ t_2) , A(  t+ t_1) \right] \dd   t_2 \dd   t_1\right)
	\exp\!\left(\int_0^hA(t+ t_1)\dd t_1\right)\\
	\nonumber
	&&\hspace*{20pt}\exp\!\left(-\frac{1}{4}\int_0^h \int_0^{  t_1} \left[ A(  t+ t_2) , A(  t+ t_1) \right] \dd   t_2 \dd   t_1\right)z(t),
\end{eqnarray}
where the outer and inner components require different numerical treatment.  Detailed estimate of the error appearing in splitting (\ref{4th_Strang}) is provided in Subsection \ref{STRANG_CALCULATIONS}.

\begin{remark}
	Approximation of a similar type
	\begin{eqnarray}\label{symmetric_Fer}
		z(t+h)&\approx&
		\exp\!\left(\frac{1}{2}\int_0^hA(t+ t_1)\dd t_1\right)
		\exp\!\left(-\frac{1}{2}\int_0^h \int_0^{  t_1} \left[ A(  t+ t_2) , A(  t+ t_1) \right] \dd   t_2 \dd   t_1\right)\\ \nonumber
		&&\hspace*{20pt}\exp\!\left(\frac{1}{2}\int_0^hA(t+ t_1)\dd t_1\right)z(t),
	\end{eqnarray}
	can be also obtained directly from the symmetric Fer expansion, which has been proposed in \cite{zanna2001fer} and further analysed and improved in \cite{ blanes2002high}. The difference between (\ref{4th_Strang}) and (\ref{symmetric_Fer}) is that the latter requires two evaluations of $\exp\!\left(\frac{1}{2}\int_0^hA(t+ t_1)\dd t_1\right)$ which, as will be obvious from the next two subsections, calls for more attention and is computationally more costly. 
\end{remark}

\subsection{Numerical treatment of the outer term}\label{outer_term}

The argument of the outer term in (\ref{4th_Strang}) reduces to a diagonal matrix
\begin{align*}
	\int_0^h \int_0^{t_1}& [A(t+ t_2),A(t+ t_1)]\dd t_2\dd t_1 \\
	\nonumber
	& = \int_0^h \int_0^{t_1}\left[ \left[  \begin{array}{cc}
		0 & 1 \\
		N(t+ t_2) & 0 
	\end{array} \right],
	\left[  \begin{array}{cc}
		0 & 1 \\
		N(t+ t_1) & 0 
	\end{array} \right]
	\right]\dd t_2\dd t_1\\
	\nonumber
	& = \int_0^h \int_0^{t_1}\left[  
	\begin{array}{cc}
		N(t+ t_1)-N(t+ t_2) & 0 \\
		0 & N(t+ t_2)-N(t+ t_1)
	\end{array}
	\right]\dd t_2\dd t_1\\
	\nonumber
	&=\left[  
	\begin{array}{cc}
		\int_0^h \int_0^{t_1}\left(f(t+ t_1)-f(t+ t_2)\right)\dd t_2\dd t_1 & 0 \\
		0 & \int_0^h \int_0^{t_1}\left(f(t+ t_2)-f(t+ t_1)\right)\dd t_2\dd t_1
	\end{array}
	\right],
\end{align*}
whose exponent can be computed numerically with low computational cost.
\begin{remark}
	By simple integration by parts we can observe, that the problem of two-dimensional quadratures boils down to the far less computationally costly one-dimensional quadrature:
	$$
	\int_0^h \int_0^{t_1} \left( f(t+ t_1)-f( t+t_2) \right)\dd t_2\dd t_1=2\int_0^h( t_1-\frac{h}{2})f(t+ t_1)\dd t_1.
	$$
\end{remark}

\begin{remark}
	We note that after semi-discretization, ${\xx}=(x^1_1,\ldots,x^1_M)\otimes\ldots\otimes(x^d_1,\ldots,x^d_M)$, the  integral under consideration, $\int_0^h \int_0^{t_1}[A(t+ t_2),A(t+ t_1)]\dd t_2\dd t_1$, can be approximated by a diagonal matrix, as a tensor product of diagonal matrices 
	\begin{eqnarray*}
		&&\left[  
		\begin{array}{cccccc}
			I(x^1_1,t) & \ldots & 0 & 0 & \ldots & 0 \\
			\vdots & \ddots & \vdots &   \vdots & \ddots & \vdots \\
			0 & \ldots& I(x^1_M,t) & 0 & \ldots& 0 \\
			0 & \ldots & 0 &  -I(x^1_1,t) & \ldots & 0 \\
			\vdots & \ddots & \vdots &  \vdots & \ddots & \vdots  \\
			0 & \ldots& 0 & 0 & \ldots& -I(x^1_M,t)
		\end{array}
		\right]\\
		&&\hspace*{20pt}\otimes\ldots\otimes
		\left[  
		\begin{array}{cccccc}
			I(x^d_1,t) & \ldots & 0 & 0 & \ldots & 0 \\
			\vdots & \ddots & \vdots &   \vdots & \ddots & \vdots \\
			0 & \ldots& I(x^d_M,t) & 0 & \ldots& 0 \\
			0 & \ldots & 0 &  -I(x^d_1,t) & \ldots & 0 \\
			\vdots & \ddots & \vdots &  \vdots & \ddots & \vdots  \\
			0 & \ldots& 0 & 0 & \ldots& -I(x^d_M,t)
		\end{array}
		\right],
	\end{eqnarray*}
	where
	$$
	I(x^k_i,t)=2 \int_0^h(t_1-\frac{h}{2})f(x^k_i,t+ t_1) \dd t_1,\quad k=1,\ldots,d, i=1,\ldots,M.
	$$
	Obviously taking the exponential of a diagonal matrix is computationally straightforward.
\end{remark}

\subsection{Numerical treatment of the inner term}\label{inner_term}

In this subsection we will tackle the inner term
\begin{eqnarray}\label{inner}
	\exp\!\left(\int_0^hA(t+ t_1)\dd t_1\right)=
	\exp\!\left( \left[ \begin{array}{cc}
		0 &  h \\
		D+ F 
		& 0 
	\end{array}\right] \right)\!,
\end{eqnarray}
where for sake of clarity we denote $D:=h\Delta$ and $F:=   \int_0^h f(t+ t_1)\dd t_1$.
Exponent of the anti-diagonal matrix (\ref{inner}) can be computed in terms of hyperbolic functions,
$$
\exp\!\left(
\left[  \begin{array}{cc}
	0 & h \\
	D+F & 0 
\end{array} \right]\right)=
\left[
\begin{array}{cc}
	\cosh \left(\sqrt{h(D+F)}\right) & \displaystyle \sqrt{\frac{h}{D+F}}  \sinh \left(\sqrt{h(D+F)}\right) \\ 
	\displaystyle \sqrt{\frac{D+F}{h}} \sinh \left(\sqrt{h(D+F)}\right) & \cosh \left(\sqrt{h(D+F)}\right) \\
\end{array}
\right]\!.
$$

The computation of hyperbolic sine or cosine of $\sqrt{h(D+F)}$ might be complicated and costly, because it is neither diagonal nor circulant symmetric.
Our aim is thus to decompose the matrix
$\left[  \begin{array}{cc}
	0 & h \\
	D+F & 0 
\end{array} \right]$
so that each of the split components can be exponentiated with low computational cost. We will use  fourth order  compact splittings of \cite{chin2002gradient} type tailored for our case,

\begin{eqnarray}\label{Splitting4a}
	\ee^{hX+hY}&\approx&\ee^{\frac16 hX}\ee^{\frac12 hY}\ee^{\frac23 hX+\frac{1}{72}h^3[X,[Y,X]]}\ee^{\frac{1}{2}hY}\ee^{\frac{1}{6}hX},
\end{eqnarray}
where the order of convergence is governed by the leading truncated terms of the symmetric Baker--Campbell-Hausdorff formula,  the fourfold nested commutators $[Y,[Y,[Y,[Y,X]]]]$, $[X,[X,[X,[X,Y]]]]$, $[X,[Y,[Y,[Y,X]]]] $, $ [Y,[X,[X,[X,Y]]]]$, $[Y,[X,[Y,[X,Y]]]]$ and $[X,[Y,[X,[Y,X]]]]$, which will be discussed in Subsection \ref{CC_CALCULATIONS}). 

\subsubsection{The inner term - towards the scheme $\Gamma^{[4]}_1$}\label{inner_term_a}

In our first  compact splitting we separate  the Laplacian  
$\left[  \begin{array}{cc}
	0 & 0 \\
	D & 0 
\end{array} \right]$
from the potential part
$
\left[  \begin{array}{cc}
	0 & h \\
	F & 0 
\end{array} \right]
$
and apply (\ref{Splitting4a}) concluding with the following splitting of order four,
\begin{eqnarray*}
	&&\exp\!\left(
	\left[  \begin{array}{cc}
		0 & h \\
		D + F & 0 
	\end{array} \right]\right)=
	\exp\!\left(
	\left[  \begin{array}{cc}
		0 & 0 \\
		D & 0 
	\end{array} \right] +
	\left[  \begin{array}{cc}
		0 & h \\
		F & 0 
	\end{array} \right]
	\right)\\
	&=&
	\exp\!\left(
	\left[  \begin{array}{cc}
		0 & 0 \\
		\frac{1}{6}D & 0 
	\end{array} \right]\right)
	\exp\!
	\left(\left[ \begin{array}{cc}
		0 & \frac{1}{2}h \\
		\frac{1}{2}F& 0 
	\end{array} \right]\right)
	\exp\!\left(
	\left[  \begin{array}{cc}
		0 & 0 \\
		\frac{2}{3}D+\frac{2}{72}hD^2 & 0 
	\end{array} \right]\right)
	\exp\!\left(
	\left[  \begin{array}{cc}
		0 & \frac{1}{2}h \\
		\frac{1}{2}F & 0 
	\end{array} \right]\right)\\
	&&\hspace*{20pt}\mbox{}\times
	\exp\!\left(
	\left[  \begin{array}{cc}
		0 & 0 \\
		\frac{1}{6}D & 0 
	\end{array} \right]\right)+\mathcal{O}(h^5).
\end{eqnarray*}
We observe, that hyperbolic functions of $\sqrt{hF}/2$ appearing in
\begin{eqnarray*}
	\exp\!\left(
	\left[  \begin{array}{cc}
		0 & \frac{1}{2}h \\
		\frac{1}{2}F & 0 
	\end{array} \right]\right)&=&
	\left[
	\begin{array}{cc}
		\displaystyle\cosh\! \left(\frac{\sqrt{hF}}{2}\right) & \displaystyle \sqrt{\frac{h}{F}} \sinh\! \left(\frac{\sqrt{hF}}{2}\right) \\ 
		\displaystyle\sqrt{\frac{F}{h}} \sinh\! \left(\frac{\sqrt{hF}}{2}\right)& \displaystyle\cosh\! \left(\frac{ \sqrt{hF}}{2}\right) \\
	\end{array}
	\right]\\
	&=&\left[
	\begin{array}{cc}
		\displaystyle\cosh\! \left(\frac{\sqrt{hF}}{2}\right) & \displaystyle\frac{h}{2} \mathrm{sinc}\! \left(\frac{\sqrt{hF}}{2}\right) \\ 
		\displaystyle\sqrt{\frac{F}{h}} \sinh\! \left(\frac{ \sqrt{hF}}{2}\right) & \displaystyle\cosh\! \left(\frac{ \sqrt{hF}}{2}\right) \\
	\end{array}
	\right] 
\end{eqnarray*}
can be computed easily because $\frac{\sqrt{hF}}{2}$ becomes a diagonal matrix following semidiscretization. Also the other two matrices can be exponentiated straightforwardly, since
$$
\exp\!\left(
\left[  \begin{array}{cc}
	0 & 0 \\
	\frac{1}{6}D & 0 
\end{array} \right]\right)
=
\left[  \begin{array}{cc}
	1 & 0 \\
	\frac{1}{6}D & 1 
\end{array} \right]
\quad {\rm and} \quad
\exp\!\left(
\left[  \begin{array}{cc}
	0 & 0 \\
	\frac{2}{3}D+\frac{2}{72}hD^2 & 0 
\end{array} \right]\right)
=
\left[  \begin{array}{cc}
	1 & 0 \\
	\frac{2}{3}D+\frac{2}{72}hD^2 & 1 
\end{array} \right]\!.
$$

\subsubsection{Inner term - towards the scheme $\Gamma^{[4]}_2$}\label{inner_term_b}

An alternative splitting may be obtained by keeping the kinetic and potential parts together, for which the  sum $D+F$ is the only nonzero entry of the matrix. After applying the splitting (\ref{Splitting4a}) we need to exponentiate matrices with  only one nonzero input,
\begin{eqnarray*}
	\exp\!\left(
	\left[  \begin{array}{cc}
		0 & h \\
		D + F & 0 
	\end{array} \right]\right)&=&
	\exp\!\left(
	\left[  \begin{array}{cc}
		0 & 0 \\
		D+F & 0 
	\end{array} \right]+
	\left[  \begin{array}{cc}
		0 & h \\
		0 & 0 
	\end{array} \right]\right)\\
	&=&
	\exp\!\left(
	\left[  \begin{array}{cc}
		0 & 0 \\
		\frac{1}{6}(D+F) & 0 
	\end{array} \right]\right)
	\exp\!\left(
	\left[  \begin{array}{cc}
		0 & \frac{1}{2}h \\
		0& 0 
	\end{array} \right]\right)\\
	&&\hspace*{20pt}\mbox{}\times
	\exp\!\left(
	\left[  \begin{array}{cc}
		0 & 0 \\
		\frac{2}{3}(D+F)+\frac{2}{72}h(D+F)^2 & 0 
	\end{array} \right]\right)
	\exp\!\left(
	\left[  \begin{array}{cc}
		0 & \frac{1}{2}h \\
		0& 0 
	\end{array} \right]\right)\\
	&&\hspace*{20pt}\mbox{}\times
	\exp\!\left(
	\left[  \begin{array}{cc}
		0 & 0 \\
		\frac{1}{6}(D+F) & 0 
	\end{array} \right]\right)+\mathcal{O}(h^5),
\end{eqnarray*}
which makes the scheme computationally extremely fast.

\subsection{Complete numerical schemes  $\Gamma_1^{[4]}$ and  $\Gamma^{[4]}_2$}

Now, having taken care of the numerical treatment of the outer term in Section \ref{outer_term} and inner term in Subsections \ref{inner_term_a} and \ref{inner_term_b}, we are ready to build up upon the local approximation presented in (\ref{4th_Strang}). Adopting  simplifications in the notation introduced earlier, we define
$$
D:=h\Delta ,\quad {\rm }\quad F_k:=   \int_0^h f(t_k+ t_1)\dd t_1\quad {\rm and}\ \quad   \mathcal{F}_k := \dfrac{1}{2} \int_0^h(t_1-\frac{h}{2})f(t_k+ t_1)\dd t_1.
$$
Assuming that the solution $ z(t_k) $ is known at $ t_k = kh, $ the scheme $\Gamma^{[4]}_1$ reads
\begin{eqnarray*}
	z(t_k+h) &\!\!\! \approx\!\!\! & \left[  \begin{array}{cc}
		\exp\left( -\mathcal{F}_k \right)  & 0 \\
		0 & \exp \left( \mathcal{F}_k \right) 
	\end{array} \right]\! \left[  \begin{array}{cc}
		1 & 0 \\
		\frac{1}{6}D & 1 
	\end{array} \right]\! \left[
	\begin{array}{cc}
		\displaystyle \cosh\! \left(\frac{\sqrt{hF_k }}{2}\right) & \displaystyle\frac{h}{2} \rm{sinc} \!\left(\frac{\sqrt{hF_k }}{2}\right) \\ 
		\displaystyle\sqrt{\frac{F_k}{h} } \sinh\! \left(\frac{ \sqrt{hF_k }}{2}\right)& \displaystyle\cosh\! \left(\frac{ \sqrt{hF_k }}{2}\right) 
	\end{array}
	\right] \!\\
	&&\hspace*{15pt}\mbox{}\times \left[  \begin{array}{cc}
		1  & 0 \\
		\frac{2}{3}D+\frac{2}{72}hD^2 & 1  
	\end{array} \right]\!\left[
	\begin{array}{cc}
		\displaystyle\cosh \left(\frac{\sqrt{hF_k }}{2}\right) & \displaystyle\frac{h}{2} \rm{sinc} \left(\frac{\sqrt{hF_k }}{2}\right) \\ 
		\displaystyle\sqrt{\frac{F_k }{h}} \sinh \left(\frac{ \sqrt{hF_k }}{2}\right) & \displaystyle\cosh \left(\frac{ \sqrt{hF_k }}{2}\right) \\
	\end{array}
	\right] \\ 
	&&\hspace*{15pt}\mbox{}\times  
	\left[  \begin{array}{cc}
		1 & 0 \\
		\frac{1}{6}D & 1 
	\end{array} \right]\!\left[  \begin{array}{cc}
		\exp\left( -\mathcal{F}_k  \right)  & 0 \\
		0 & \exp \left( \mathcal{F}_k \right) 
	\end{array} \right]z(t_k).
\end{eqnarray*}

The algorithm for the scheme $ \Gamma^{[4]}_1 $ is presented in Table \ref{TABLE_1}.

\begin{center} \colorbox{light-gray}{
		\begin{tabular}{l}
			$ T $ time steps of 4th order algorithm 	$ \Gamma^{[4]}_1 $\\ 
			\hline \\
			\textbf{do }$  k=0,T-1 $ \\ \smallskip
			$\qquad q_0 = \exp(- \mathcal{F}_k)  z_1  ; \quad  p_0 =\exp( \mathcal{F}_k)  z_2$ \\ \smallskip
			$\qquad \qquad p_1 = \frac{1}{6} D q_0 + p_0 $ \\ \smallskip
			$\qquad \qquad q_1 = \cosh \!\left(\frac{ \sqrt{h F_k}}{2}\right)  q_0 +\frac{h}{2}\rm{sinhc} \!\left(\frac{\sqrt{hF_k }}{2}\right) p_1 $ \\ \smallskip
			$\qquad \qquad \qquad p_2 =\sqrt{\frac{F_k}{h}}  \sinh \!\left(\frac{ \sqrt{h F_k}}{2}\right) q_0 + \cosh\! \left(\frac{ \sqrt{h F_k}}{2}\right) p_1$\\ \smallskip
			$ \qquad \qquad \qquad \qquad p_3 = \left(  \frac{2}{3}D+\frac{2}{72}hD^2\right) q_1 + p_2  $\\ \smallskip
			$\qquad \qquad \qquad q_2 = \cosh\! \left(\frac{ \sqrt{h F_k}}{2}\right)  q_1 +  \frac{h}{2}\rm{sinhc}\! \left(\frac{\sqrt{hF_k }}{2}\right)   p_3 $ \\ \smallskip
			$ \qquad\qquad p_4 = 
			\sqrt{\frac{F_k}{h}} \sinh\! \left(\frac{ \sqrt{h F_k}}{2}\right) q_1+ \cosh\! \left(\frac{ \sqrt{h F_k}}{2}\right) p_3$\\ \smallskip
			$\qquad \qquad p_5 = \frac{1}{6} D q_2 + p_4 $ \\ \smallskip
			$\qquad \qquad q_3 = \exp(- \mathcal{F}_k) q_2  ;\quad  p_6 =\exp( \mathcal{F}_k) p_5$ \\ \smallskip
			$\qquad z_1 := q_3 ; \quad z_2 := p_6  $\\
			\textbf{end do} \\
			\hline 
	\end{tabular}}
	\captionof{table}{The algorithm $ \Gamma^{[4]}_1 $ for finding the approximate solution on the time interval $[t_0,t_T]$ with $T$ time steps $h=(t_T-t_0)/T$. Note that $z(t)=[z_1(t),z_2(t)]^\top$, where  the discretisation in space is neither applied nor, indeed,  specified.}
	\label{TABLE_1}
\end{center}

In a similar vain we present final scheme $\Gamma^{[4]}_2$, and its algorithm is presented in Table \ref{TABLE_2}.
\begin{eqnarray*}
	z(t_k+h) &\!\!\! \approx\!\!\! & \left[  \begin{array}{cc}
		\exp\left(- \mathcal{F}_k \right)  & 0 \\
		0 & \exp \left( \mathcal{F}_k \right) 
	\end{array} \right] \!
	\left[  \begin{array}{cc}
		1 & 0 \\
		\frac{1}{6}(D+F_k) & 1
	\end{array} \right]\!
	\left[  \begin{array}{cc}
		1 & \frac{1}{2}h \\
		0& 1
	\end{array} \right]\\
	&&\mbox{}\times 
	\left[  \begin{array}{cc}
		1 & 0 \\
		\frac{2}{3}(D+F_k)+\frac{2}{72}h(D+F_k)^2 & 1
	\end{array} \right] \\
	&&\mbox{}\times \left[  \begin{array}{cc}
		1 & \frac{1}{2}h \\
		0& 1
	\end{array} \right]\!
	\left[  \begin{array}{cc}
		1 & 0 \\
		\frac{1}{6}(D+F_k) & 1
	\end{array} \right] \left[  \begin{array}{cc}
		\exp\left( - \mathcal{F}_k \right)  & 0 \\
		0 & \exp \left(\mathcal{F}_k \right) 
	\end{array} \right] \!z(t_k).
\end{eqnarray*} 

\begin{center} \colorbox{light-gray}{
		\begin{tabular}{l}
			$ T $ time steps of 4th order algorithm 	$ \Gamma^{[4]}_2 $\\ 
			\hline \\
			\textbf{do }$  k=0,T-1 $ \\ \smallskip
			$\qquad q_0 = \exp(-\mathcal{F}_k)  z_1  ; \quad  p_0 =\exp( \mathcal{F}_k)  z_2$ \\ \smallskip
			$  \qquad p_1 = \frac{1}{6} (D+F_k) q_0 + p_0 $ \\ \smallskip
			$ \qquad \qquad q_1 = q_0 +\frac{1}{2} h p_1 $ \\ \smallskip
			$ \qquad \qquad \qquad p_2 = \left( \frac{2}{3}(D+F_k)+\frac{2}{72}h(D+F_k)^2 \right) q_1+p_1; $ \\ \smallskip
			
			$ \qquad \qquad q_2 = q_1+\frac{1}{2} h p_2 $ \\ \smallskip
			
			$ \qquad \qquad p_3= \frac{1}{6} (D+F_k) q_2 + p_2 $ \\ \smallskip
			$ \qquad q_3 = \exp(-\mathcal{F}_k)  q_2  ; \quad  p_4 =\exp( \mathcal{F}_k) p_3 $ \\ \smallskip
			
			$\qquad z_1 := q_3 ; \quad z_2 := p_4  $\\ \smallskip
			\textbf{end do} \\
			\hline 
	\end{tabular}}
	\captionof{table}{Algorithm $ \Gamma^{[4]}_2 $ for finding the approximate solution on the time interval $[t_0,t_T]$ with $T$ time steps $h=(t_T-t_0)/T$. Note that $z(t)=[z_1(t),z_2(t)]^\top$, where  the discretisation in space is  neither applied nor specified.}
	\label{TABLE_2}
\end{center}

\section{An estimate of  leading error terms}\label{ESTIMATES}

In the previous section we presented the full derivation of the schemes arguing that all leading error terms scale like $\mathcal{O}(h^5)$. We mentioned, however, that some of the principal error terms of Magnus expansions and of Strang splitting may depend on the oscillatory coefficients $\omega_n$. In this Section we provide careful estimates of these terms. The following theorem will be frequently exploited in these estimates.

\begin{theorem}\label{thm1}
	Let $ a \in C^1[0,h] $ be a real function,  $ h \leq 1  $ and $ \omega \geq 1 $. 	Then the following estimate holds:
	\begin{eqnarray}\label{est_thm}
		\left|\int_0^h \int_0^{t_1} \int_0^{t_2} .... \int_0^{t_{m-1} } a(t_k) \ee^{\ii  \omega t_k }  \dd t_m \dd t_{m-1} ... dt_1  \right| \leq C \: \mathbb{A}  \min\! \left\lbrace h^{m}, \dfrac{h^{m-1}}{\omega} \right\rbrace , \quad 1 \leq k\leq m,
	\end{eqnarray}
	where $C$ is a constant and  	$ \mathbb{A}  = \max_{t \in [0,h]} \left\lbrace |a(t)|, |a'(t)|\right\rbrace $.
\end{theorem}
\begin{proof}
	Let us understand the $C$ is a generic constant and start with an observation that for every $r=0,1,2,\ldots$
	\begin{eqnarray}\label{est}
		\left| \int_0^{t_{m-1}} t^r_m a(t_m) \ee^{\ii  \omega t_m}  \dd t_m \right| & \leq &(r+2) C \mathbb{A}  \min \!\left\lbrace h^{r+1},\dfrac{h^r}{\omega} \right\rbrace\!, \quad 0 \leq  t_{m} < t_{m-1} \leq h.
	\end{eqnarray}
	Indeed,  an immediate estimate is
	\begin{eqnarray}\label{est_1}
		\left| \int_0^{t_{m-1}} t^r_m a(t_m) \ee^{\ii  \omega t_m}  \dd t_m \right| & \leq &C \mathbb{A}  h^{r+1},
	\end{eqnarray}
	but by simple integration by parts we observe that for $r=0$
	\begin{eqnarray}
		\nonumber
		\int_0^{t_{m-1}}  a(t_m) \ee^{\ii  \omega t_m}  \dd t_m & =& \dfrac{1}{\ii \omega}  a(t_{m-1})\ee^{\ii \omega t_{m-1}}  - \frac{1}{\ii \omega}a(0) - \dfrac{1}{\ii \omega} \int_0^{t_{m-1}}  a'(t_m) \ee^{\ii  \omega t_m}  \dd t_m, 
	\end{eqnarray}
	and that for $r=1,2,\ldots$ 
	\begin{eqnarray}
		\nonumber
		\int_0^{t_{m-1}} t^r_m a(t_m) \ee^{\ii  \omega t_m}  \dd t_m & =& \dfrac{1}{\ii \omega} t_{m-1}^{r} a(t_{m-1})\ee^{\ii \omega t_{m-1}} - \dfrac{1}{\ii \omega} \int_0^{t_{m-1}} r t_m^{r-1} a(t_m) \ee^{\ii  \omega t_m}  \dd t_m \\
		&-& \dfrac{1}{\ii \omega} \int_0^{t_{m-1}} t_m^r a'(t_m) \ee^{\ii  \omega t_m}  \dd t_m 
	\end{eqnarray}	
	and can obtain the more subtle result,
	\begin{eqnarray}\label{est_2}
		\left| \int_0^{t_{m-1}} t^r_m a(t_m) \ee^{\ii  \omega t_m}  \dd t_m \right| & \leq & \dfrac{h^r}{\omega} C\mathbb{A}  +r \dfrac{h^r}{\omega} C \mathbb{A} +  \dfrac{h^{r+1}}{\omega} C\mathbb{A}, \quad r=0,1,2,\ldots
	\end{eqnarray}
	Combining (\ref{est_1}) and (\ref{est_2}) we  conclude with the inequality (\ref{est}). Finally, inequality (\ref{est_thm}) is an  immediate consequence of (\ref{est}).
\end{proof}

\begin{remark}\label{sum_over_n}
	Let $ a_n \in C^1[0,h], \; n < |N|, N \in \mathbb{N} $ be a family of known,  real-valued functions and let   $ h \leq 1  $ and $  \omega_{\min}  \geq 1 $. 	Then we can show a result similar to the one obtained in Theorem \ref{thm1}, namely, that
	\begin{eqnarray}\label{est_remark}
		\left|\int_0^h \int_0^{t_1} \int_0^{t_2} .... \int_0^{t_{m-1} } \sum_{|n|\leq N} a_n(t_k) \ee^{\ii   \omega_n t_k }  \dd t_m \dd t_{m-1} ... dt_1  \right| \leq C \: \mathbb{A}  \min \!\left\lbrace h^{m}, \dfrac{h^{m-1}}{ \omega_{\min}} \right\rbrace\! , \quad 1 \leq k\leq m,
	\end{eqnarray}
	where $C$ is a generic constant and  	$ \mathbb{A}  = \max_{t \in [0,h]} \left\lbrace |a_n(t|, |a_n'(t)|; |n|<N \right\rbrace $.
\end{remark}

\begin{remark}
	For   clarity of exposition we will abuse the notation and write 
	\begin{eqnarray}\label{high_oscyl_1}
		\nonumber
		f(t_k) = \alpha(t_k) + \sum_{|n|\leq N} a_n(t_k) \ee^{i  \omega_n  t_k}
	\end{eqnarray}
	instead of 
	\begin{eqnarray}
		\nonumber
		f(t+t_k)   = \alpha(t+t_k) + \sum_{|n|\leq N} a_n (t+t_k) \ee^{i  \omega_n  (t+t_k)}.
	\end{eqnarray} 
\end{remark}

In the remainder of the paper we denote by $\Vert \cdot\Vert_s$ the standard $H^s$ Sobolev norm on the torus $\mathbb{T}^d$ and assume that that the following assumption holds.

\begin{assumption}\label{MainAssumption}
	Seeking solutions $\psi(t):=\psi(\cdot,t) \in C^2([0,T],H^{s+6}(\mathbb{T}^d))$, we assume that 
	\begin{enumerate}
		\item
		$a_n(t):=a_n(\cdot,t) \in C^2([0,T],H^{s+4}(\mathbb{T}^d))$,
		\item
		$\alpha(t):=\alpha(\cdot,t) \in C^2([0,T],H^{s+4}(\mathbb{T}^d))$,
		\item
		$s\geq0$, $s+6>d/2$.
	\end{enumerate}
\end{assumption}
\begin{definition}\label{errors_def}
	In all calculations of this article we denote by $ C $ a generic constant. Furthermore, we let
	\begin{align*}
		\error = & \min \left\lbrace h^{3},\dfrac{h^2}{ \omega_{\min}},  h^5\omega^2_{\max}  \right\rbrace\!, \\
		\widetilde{ \mathcal{A}} = & \max_{t \in  [0,T]} \left\lbrace \left\| a_n(t) \right\|_{s+4},\left\|a_n^2(t) \right\|_{s+2},\left\|a_n^3(t) \right\|_{s}, \left\| \partial_t a_n(t) \right\|_{s+2}  , \left\| \partial_t^2 a_n(t) \right\|_{s+2};\ \vert n \vert \leq N\right\rbrace\!,  \\ 
		\nonumber
		\widetilde{ \mathcal{L}} = & \max_{t \in  [0,T]} \left\lbrace \left\| \alpha(t) \right\|_{s+4},\left\|\alpha^2(t) \right\|_{s+2},\left\|\alpha^3(t) \right\|_{s}, \left\| \partial_t \alpha(t) \right\|_{s+2}  , \left\| \partial_t^2 \alpha(t) \right\|_{s+2}\right\rbrace\! .
	\end{align*}
\end{definition}

\subsection{Leading error terms of Magnus expansion truncation}\label{MAGNUS_CALCULATIONS}

The commutators appearing in (\ref{mag3})--(\ref{mag4}) are calculated  in Appendix \ref{apA}. We commence with
\begin{eqnarray}
	\nonumber
	\Theta_3(t+h,t) & =&\dfrac{1}{6} \int_0^h \int_0^{  t_1} \int_0^{  t_2}  \left[  \begin{array}{cc}
		0 & H_1+H_3 \\
		H_2+H_4 & 0
	\end{array} \right] \dd   t_3  \dd   t_2 \dd   t_1,
\end{eqnarray}
where 
\begin{eqnarray}
	H_1 + H_3 & =& 2 \left[ - f(  t_1) + 2f(  t_2) - f(  t_3) \right];  \\
	\nonumber
	H_2 + H_4 & =&  \Delta  \left[ f(t_3) -2 f(t_2)+ f(t_1)\right] + \left[ f(t_3) -2 f(t_2)+ f(t_1)\right] \Delta  \\
	\nonumber
	&& \mbox{}+4 f(t_1) f(t_3) - 2 f(t_1) f(t_2)   - 2 f(t_3) f(t_2).
\end{eqnarray} 

In the estimates below  we separate the non-oscillatory  from the oscillatory parts and obtain
\begin{align}\nonumber
	\left| 	\int_0^h \int_0^{  t_1} \int_0^{  t_2} \right. &\left. \left( H_1 + H_3 \right)  \dd   t_3  \dd   t_2 \dd   t_1 \right|  \leq   \left| \int_0^h \int_0^{  t_1} \int_0^{  t_2} 2 \left[ - \alpha(t_1) + 2\alpha (t_2) - \alpha (t_3) \right]  \dd   t_3  \dd   t_2 \dd   t_1 \right| \\
	\nonumber
	&\mbox{}+ \sum_{|n|\leq N}  \left| \int_0^h \int_0^{  t_1} \int_0^{  t_2} 2 \left[ - a_n(t_1)\ee^{i \omega_n t_1} + 2a_n(t_2) \ee^{i \omega_n t_2} - a_n(t_3)\ee^{i \omega_n t_3} \right]  \dd   t_3  \dd   t_2 \dd   t_1 \right| \\
	\label{h1h3}
	&\leq Ch^5  \max_{t \in [0,h]} | \alpha''(t)|  + C \: \widetilde{ \mathcal{A}} \min \left\lbrace h^{3},\dfrac{h^2}{ \omega_{\min}},  h^5\omega^2_{\max}  \right\rbrace\\
	\nonumber
	&\leq  Ch^5 \widetilde{ \mathcal{L}} + C \: \widetilde{ \mathcal{A}}  \error . 
\end{align}

The estimate of the first summand may be obtained by  expanding $ \alpha(t_i), i = 1,2,3 $ into a Taylor series at the origin: there exist such $\xi_1, \xi_2, \xi_3 \in [0,h]$ that
\begin{align*}
	\bigg\vert	\int_0^h \int_0^{  t_1} \int_0^{  t_2}  &	 \left[  \alpha(  t_1) - 2\alpha(  t_2) + \alpha(  t_3) \right]   \dd   t_3  \dd   t_2 \dd   t_1\bigg\vert  \\
	&= \left\vert  \int_0^h \int_0^{  t_1} \int_0^{  t_2} 	\alpha'(0)(t_1-2t_2+t_3)\:   \dd   t_3  \dd   t_2 \dd   t_1 + \int_0^h \int_0^{  t_1} \int_0^{  t_2} 	t_1^2\alpha''(\xi_1) -2t_2^2\alpha''(\xi_2)+t_3^2\alpha''(\xi_3)\:    \dd   t_3  \dd   t_2 \dd   t_1  \right\vert \\
	& \leq 4 h^5 \max_{t \in [0,T]} | \alpha''(t)|,
\end{align*}  
because the first triple integral vanishes. 

More attention should be paid to the case, where Laplace operator $\Delta$  appears in our expression.  We need to bear in mind that it is applied not just to  the function $ f$. 
Specifically, we are estimating components of the operator $\Theta$, so $\Delta $ applies not only to $f$, but also to the solution $\psi$, and this is the reason why we need to  raise regularity requirements from $\|\cdot\|_s$ to $\|\cdot\|_{s+2}$. 
More precisely, for any $\varphi$ sufficiently smooth we have in the standard $H^s$ Sobolev norm    
\begin{align*}
	&\left \Vert	\int_0^h \int_0^{  t_1} \int_0^{  t_2}  (H_2 + H_4 ) \dd   t_3  \dd   t_2 \dd   t_1 \varphi\right\Vert_s  \leq \left\Vert    \int_0^h \int_0^{  t_1} \int_0^{  t_2} 	\Delta \left[  f(  t_1) - 2f(  t_2) + f(  t_3) \right]   \dd   t_3  \dd   t_2 \dd   t_1 \varphi \right\Vert_s \\
	\nonumber
	&\quad \quad  \quad+ \left\Vert \int_0^h \int_0^{  t_1} \int_0^{  t_2} 	\left[  f(  t_1) - 2f(  t_2) + f(  t_3) \right]\Delta    \dd   t_3  \dd   t_2 \dd   t_1  \varphi  \right\Vert_s \\
	\nonumber
	&\quad \quad \quad+ \left\Vert \int_0^h \int_0^{  t_1} \int_0^{  t_2} 	\left[4 f(  t_1) f(  t_3) -2  f(  t_1) f(  t_2) -2  f( t_3) f( t_2)\right] \dd   t_3  \dd   t_2 \dd   t_1  \varphi \right\Vert_s\\
	\nonumber
	& \quad \quad \quad \leq   C \: h^5 \widetilde{ \mathcal{L}}\Vert \varphi\Vert_{s+2} + C \: \widetilde{ \mathcal{A}} \min \left\lbrace h^{3},\dfrac{h^2}{ \omega_{\min} }, h^5\omega^2_{\max}  \right\rbrace  \Vert \varphi\Vert_{s+2} +  C  \: \widetilde{ \mathcal{L}}  \: \widetilde{ \mathcal{A}} \min \left\lbrace h^{3},\dfrac{h^2}{ \omega_{\min} }, h^5\omega^2_{\max}   \right\rbrace  \Vert \varphi\Vert_{s}  \\
	& \quad \quad \quad \leq   C \: h^5 \widetilde{ \mathcal{L}} \: \Vert \varphi\Vert_{s+2} + C \: \widetilde{ \mathcal{A}} \: \error \: \Vert \varphi\Vert_{s+2} +  C  \: \widetilde{ \mathcal{L}}  \: \widetilde{ \mathcal{A}} \: \error  \Vert \varphi\Vert_{s}
\end{align*}
To explicate further this inequality, we estimate one of the integrands below 
\begin{align*}
	\bigg\|  \int_0^h \int_0^{  t_1} \int_0^{  t_2} & \Delta 	\left(  f(  t_1) - 2f(  t_2) + f(  t_3) \right)\dd   t_3  \dd   t_2 \dd   t_1 \varphi  \bigg\| _s \\
	& \leq  C  \: \max_{t \in [0,T]} \left\lbrace h^5 \Vert  \partial_t^2 \alpha(t) \Delta \varphi \Vert_s , h^5 \Vert  \nabla \partial_t^2 \alpha(t) \nabla \varphi \Vert_s, h^5 \Vert  \Delta \partial_t^2 \alpha(t)\varphi \Vert_s \right\rbrace \\
	&\mbox{} + C\min \left\{   h^3 \: \max_{t \in [0,T]} \left\{  \Vert   a_n(t)  \Delta \varphi \Vert_s ,  \Vert \nabla  a_n(t) \nabla \varphi \Vert_s,  \Vert  \Delta a_n(t) \varphi \Vert_s  \right \}\right.,\\
	&\hspace*{15pt} \dfrac{h^2} {\omega_{\min}}   \: \max_{t \in [0,T]} \left\{  \Vert   \partial_t a_n(t)  \Delta \varphi \Vert_s ,  \Vert \nabla  \partial_t a_n(t) \nabla \varphi \Vert_s,  \Vert  \Delta \partial_t a_n(t) \varphi \Vert_s  \right\},\\
	& \left.\hspace*{15pt} h^5\omega^2_{\max}  \: \max_{t \in [0,T]} \left\{ \Vert   \partial_t^2 a_n(t)  \Delta \varphi \Vert_s ,  \Vert \nabla  \partial_t^2 a_n(t) \nabla \varphi \Vert_s,  \Vert  \Delta \partial_t^2 a_n(t) \varphi \Vert_s  \right\} \right\}.
\end{align*}

The computations of the  commutators  appearing in $\Theta_4$ are more complicated (details can be found in Appendix \ref{apA}), yet the estimates  are similar to these appearing in $\Theta_3$,
\begin{eqnarray}
	\nonumber
	\Theta_4(t+h,t) & =& \dfrac{1}{12} \int_0^h \int_0^{  t_1} \int_0^{  t_2}  \int_0^{  t_3}\left[  \begin{array}{cc}
		\mathcal{H}_1 &  0\\
		0 &  	\mathcal{H}_2
	\end{array} \right]\dd   t_4  \dd   t_3  \dd   t_2 \dd   t_1,
\end{eqnarray}
where 
\begin{eqnarray*}
	\mathcal{H}_1	&\!\!\!=\!\!\!&  \Delta  \left[  2f(  t_2)  -2 f(  t_3)  \right] +3 \left[  2f(  t_2)  -2 f(  t_3)  \right]\Delta  + 4f(  t_4) \left[  f(  t_2)  -f(  t_3)  \right] + 4f(  t_1) \left[ f(  t_2)  - f(  t_3)  \right] ;\\
	\mathcal{H}_2	& \!\!\!=\!\!\!&  -3\Delta  \left[  2f(  t_2)  -2 f(  t_3)  \right] - \left[  2f(  t_2)  -2 f(  t_3)  \right]\Delta  - 4f(  t_4) \left[  f(  t_2)  -f(  t_3)  \right] - 4f(  t_1) \left[  f(  t_2)  - f(  t_3)  \right] .
\end{eqnarray*}

Observing  that 
$$  f(  t_i)  - f(  t_j) = h (f’(\xi_i)-f’(\xi_j)), \quad {\rm for} \ 0 < t_i , t_j < h, \quad {\rm and}\ {\rm certain}\  0 < \xi_i , \xi_j < h  $$
and recalling Theorem \ref{thm1} we conclude that for sufficiently smooth $\varphi$
\begin{align*}
	\left\Vert  	\int_0^h \int_0^{  t_1} \int_0^{  t_2} \int_0^{  t_3}  \mathcal{H}_1 \dd   t_4  \dd   t_3  \dd   t_2 \dd   t_1\varphi  \right  \Vert_s
	& \leq    C \: h^5 \widetilde{ \mathcal{L}} \Vert \varphi  \Vert_{s+2} + C \: \widetilde{ \mathcal{A}} \min \left\lbrace h^{4},\dfrac{h^3}{ \omega_{\min} },h^5 \omega_{\max} \right\rbrace  \Vert \varphi  \Vert_{s+2} \\
	\nonumber
	& + C  \: \widetilde{ \mathcal{L}}  \: \widetilde{ \mathcal{A}}  \min \left\lbrace h^{4},\dfrac{h^3}{\omega_{\min}}, h^5 \omega_{\max} \right\rbrace  \Vert \varphi\Vert_{s} .
\end{align*}

The estimate of the quadruple integral of $ \mathcal{H}_2 $ follows along similar lines. \\

Components $ \Theta_k(t+h,t),\; k\geq5 $ of the Magnus expansion (\ref{6}) require $(k-4)$ more integrals of $ \mathcal{H}_1 $ and $ \mathcal{H}_2 $ which scales their magnitude by $h^{k-4}$. For this reason $ \Theta_k(t+h,t),\; k\geq5 $ do not contribute to leading error term of the truncation.

Summing up the entries of the matrix $\Theta_3(t+h,t)    + \Theta_4(t+h,t)$  constituting the leading error term committed by truncating the Magnus expansion can be bounded for any  sufficiently smooth function $\varphi$ by
$$
C  \left[ \widetilde{ \mathcal{L}}  \: \widetilde{ \mathcal{A}}  \:  \error  \Vert \varphi\Vert_{s} +  \left( \: \widetilde{ \mathcal{A}}  \error + h^5 \: \widetilde{ \mathcal{L}}\right)  \: \Vert \varphi \Vert_{s+2} \right]. 
$$

\subsection{Leading error terms of  the Strang splitting}\label{STRANG_CALCULATIONS}

Let
\begin{displaymath}
	X = \int_0^h\int_0^{t_1}[A(t_1),A(t_2)]\dd t_2 \dd t_1 =\int_0^h\int_0^{t_1}	\left[  
	\begin{array}{cc}
		f( t_1)-f(t_2) & 0 \\
		0 & f( t_2)-f(t_1)
	\end{array}
	\right]\!\dd t_2 \dd t_1,
\end{displaymath} 
\begin{displaymath}
	Y = \int_0^h \left[  \begin{array}{cc}
		0 & 1 \\
		\Delta  + f(t_1)& 0 
	\end{array} \right] \!\dd t_1,
\end{displaymath}
and consider their twofold nested commutators (calculations are presented in Appendix \ref{strang_est})

\begin{eqnarray*}
	[Y,[Y,X]] &=&  \left[  \begin{array}{cc}
		-h^2 \Delta  \mathcal{F} -3 h^2 \mathcal{F} \Delta   -  4h \mathcal{F} F	 & 0\\
		0	& 3h^2 \Delta  \mathcal{F} +h^2 \mathcal{F} \Delta   +  4h \mathcal{F} F	
	\end{array} \right]\!; \\{}
	[X,[X,Y]] &=&  \left[  \begin{array}{cc}
		0 & 4h \mathcal{F}^2\\
		2h\mathcal{F} \Delta  \mathcal{F} + h\mathcal{F}^2 \Delta  + h \Delta  \mathcal{F}^2 + 4\mathcal{F}^2 F& 0 
	\end{array} \right]\!,
\end{eqnarray*}
where
\begin{displaymath}
	F = \int_0^h f( t+ t_1) \dd t_1 \qquad {\rm and} \qquad \mathcal{F} = \int_0^h \int_0^{t_1} \left[ f( t+t_2)-f(t+ t_1) \right]\dd t_2\dd t_1.	
\end{displaymath}

Using Theorem \ref{thm1} we obtain the following estimates
\begin{eqnarray*}
	| F |  &\leq&  C \: h  \max_{t \in [0,T]} \left\lbrace  \Vert \alpha(t)\Vert_s \right\rbrace +  C \: h \max_{t \in [0,T]} \left\lbrace  \Vert  a_n(t)\Vert _s, \Vert  \partial_t a_n(t)\Vert _s \right\rbrace \leq C h(\widetilde{ \mathcal{L}} + \widetilde{ \mathcal{A}}), \\
	|\mathcal{F}|   &\leq&  C \: h^2  \max_{t \in [0,T]} \left\lbrace h \Vert \partial_t \alpha(t)\Vert_s  \right\rbrace + C \: \min \left\lbrace h^2,\dfrac{h}{\omega_{\min}}, h^3 \omega_{\max} \right\rbrace \max_{t \in [0,T]} \left\lbrace  \Vert  a_n(t)\Vert_s,   \Vert  \partial_t a_n(t)\Vert_s \right\rbrace \\
	& \leq & C\:h^3 \widetilde{ \mathcal{L}} +  C \widetilde{ \mathcal{A}}\min \left\lbrace h^2,\dfrac{h}{\omega_{\min}}, h^3 \omega_{\max} \right\rbrace, \\
	\left\| \Delta  \mathcal{F} \varphi \right\|_s  & \leq & C \: h^2 \: \max_{t \in [0,T]} \left\lbrace h \left\|  \partial_t \alpha(t)\Delta \varphi \right\|_s , h \left\|  \nabla \partial_t \alpha(t)\nabla \varphi \right\|_s, h \left\|  \Delta \partial_t \alpha(t)\varphi \right\|_s  \right\rbrace \\
	&&\mbox{}+  C \:\widetilde{ \mathcal{A}} \min  \left\lbrace h^2,\dfrac{h}{\omega_{\min}}, h^3 \omega_{\max} \right\rbrace \Vert \varphi  \Vert_{s+2} \\
	& \leq &  C \: \left( h^3 \widetilde{ \mathcal{L}} +  \widetilde{ \mathcal{A}}\min \left\lbrace h^2,\dfrac{h}{\omega_{\min}}, h^3 \omega_{\max} \right\rbrace \right) \Vert \varphi  \Vert_{s+2},\\
	\left\| \mathcal{F} \Delta   \varphi \right\|_s   &  \leq &C \:  \left( h^3 \widetilde{ \mathcal{L}} +  \widetilde{ \mathcal{A}}\min \left\lbrace h^2,\dfrac{h}{\omega_{\min}}, h^3 \omega_{\max} \right\rbrace\right) \Vert \varphi  \Vert_{s+2},\\
	\nonumber
	\left\| \mathcal{F} F \varphi  \right\|_s & \leq & C \: \left(h^3 \widetilde{ \mathcal{L}} +  \widetilde{ \mathcal{A}}\min \left\lbrace h^2,\dfrac{h}{\omega_{\min}}, h^3 \omega_{\max} \right\rbrace\right) \left( h \widetilde{ \mathcal{L}}+\widetilde{ \mathcal{A}}h\right)  \Vert \varphi  \Vert_{s}, \\
	\left\| \mathcal{F} \Delta  \mathcal{F} \varphi   \right\|_s  & \leq &C \:  \left( h^3 \widetilde{ \mathcal{L}} +  \widetilde{ \mathcal{A}}\min \left\lbrace h^2,\dfrac{h}{\omega_{\min}}, h^3 \omega_{\max} \right\rbrace\right)^2 \Vert \varphi  \Vert_{s+2}, \\
	\left\| \mathcal{F}^2 \Delta  \varphi   \right\|_s  & \leq & C \: \left( h^3 \widetilde{ \mathcal{L}} +  \widetilde{ \mathcal{A}}\min \left\lbrace h^2,\dfrac{h}{\omega_{\min}}, h^3 \omega_{\max} \right\rbrace\right)^2 \Vert \varphi  \Vert_{s+2},\\ \nonumber
	\left\| \Delta  \mathcal{F}^2  \varphi   \right\|_s & \leq & C \: \left(h^3 \widetilde{ \mathcal{L}} +  \widetilde{ \mathcal{A}}\min \left\lbrace h^2,\dfrac{h}{\omega_{\min}}, h^3 \omega_{\max} \right\rbrace\right)^2 \Vert \varphi  \Vert_{s+2}, \\
	\left\| \mathcal{F}^2 F \varphi   \right\|_s & \leq &  C \: \left(h^3 \widetilde{ \mathcal{L}} +  \widetilde{ \mathcal{A}}\min \left\lbrace h^2,\dfrac{h}{\omega_{\min}}, h^3 \omega_{\max}  \right\rbrace\right)^2 \left( h \widetilde{ \mathcal{L}}+\widetilde{ \mathcal{A}}h\right) \Vert \varphi  \Vert_{s}.
\end{eqnarray*}
Grouping all this together, we observe that the leading error term of Strang splitting  is bounded by 
$$
C \left[ \left( h^5\widetilde{ \mathcal{L}}^2 +h^5\widetilde{ \mathcal{L}} \widetilde{ \mathcal{A}}+\widetilde{ \mathcal{A}}^2 \error + \widetilde{ \mathcal{L}}  \: \widetilde{ \mathcal{A}}  \error \right)  \Vert \varphi  \Vert_{s} + \left(   h^5 \widetilde{ \mathcal{L}} +\widetilde{ \mathcal{A}}  \error  \right)  \Vert \varphi  \Vert_{s+2} \right]
$$

and that the sought accuracy, $\mathcal{O}(h^5+\min\{h^3,h^2/\omega_{\min},h^5\omega_{\max}^2\})$ and   $\mathcal{O}(\min\{h^3,h^2/\omega_{\min},h^5\omega_{\max}^2\})$ in the special case $\alpha(\xx,t)\equiv 0$, is satisfied.\\

$k$-fold nested commutators, $k\geq3$, do not contribute to the principal error term. Indeed, they involve additional multiplications by expressions of magnitude $h$ scaling the twofold nested commutators by $h^{k-3}$. This observation is consistent with the work in \cite{jahnke2000error}, where the derivation and analysis of the error bounds of Strang splitting is based only on the twofold nested commutators $[Y,[Y,X]]$ and $[X,[X,Y]]$.

\subsection{Leading error terms of  the 4th order compact splittings}\label{CC_CALCULATIONS}

Proposed methods $\Gamma^{[4]}_1$ and $\Gamma_2^{[4]}$ differ by the choice of matrices $X$ and $Y$ in  compact splittings (\ref{Splitting4a}).
Let us recall that to obtain $\Gamma^{[4]}_1$ we took 
$X=\left[  \begin{array}{cc}
	0 & 0 \\
	D & 0 
\end{array} \right]$ and
$
Y= \left[  \begin{array}{cc}
	0 & h \\
	F & 0 
\end{array} \right], 
$
where $D:=h\Delta$ and $F:=   \int_0^h f(t+ t_1)\dd t_1$.\\

It is not difficult to check that the leading error terms of these splittings satisfy

\begin{align*}
	[Y,[Y,[Y,[Y,X]]]]  &=  \left[  \begin{array}{cc}
		0 & -4h^3FD-h^3DF \\
		h^2F^2D+6h^2FDF +h^2DF^2 & 0 
	\end{array} \right] \\
	[X,[X,[X,[X,Y]]]] &=  \left[  \begin{array}{cc}
		0 & 0 \\
		0 & 0 
	\end{array} \right] \\
	[X,[Y,[Y,[Y,X]]]] &=  \left[  \begin{array}{cc}
		0 & 0 \\
		2h^2DFD+3h^2D^2F+3h^2FD^2 & 0 
	\end{array} \right]  \\
	[Y,[X,[X,[X,Y]]]] &=  \left[  \begin{array}{cc}
		0 & 0 \\
		0 & 0 
	\end{array} \right]  \\
	[Y,[X,[Y,[X,Y]]]]  &=  \left[  \begin{array}{cc}
		0 & -4h^3D^2 \\
		2h^2(FD^2+D^2F) & 0 
	\end{array} \right]  \\
	[X,[Y,[X,[Y,X]]]] &=  \left[  \begin{array}{cc}
		0 & 0 \\
		4h^2 D^3& 0 
	\end{array} \right]  \\
\end{align*}

In the case of $\Gamma^{[4]}_2$ we took 
$X=\left[  \begin{array}{cc}
	0 & 0 \\
	D+F & 0
\end{array} \right]$ and
$
Y= \left[  \begin{array}{cc}
	0 & h \\
	0 & 0 
\end{array} \right], 
$
obtaining
\begin{align*}
	[Y,[Y,[Y,[Y,X]]]]  &=  \left[  \begin{array}{cc}
		0 & 0 \\
		0 & 0 
	\end{array} \right] \\
	[X,[X,[X,[X,Y]]]] &=  \left[  \begin{array}{cc}
		0 & 0 \\
		0 & 0 
	\end{array} \right]  \\
	[X,[Y,[Y,[Y,X]]]] &=  \left[  \begin{array}{cc}
		0 & 0 \\
		0 & 0 
	\end{array} \right]  \\
	[Y,[X,[X,[X,Y]]]] &=  \left[  \begin{array}{cc}
		0 & 0 \\
		0 & 0 
	\end{array} \right]  \\
	[Y,[X,[Y,[X,Y]]]]  &=  \left[  \begin{array}{cc}
		0 & 4h^3(D+F)^2 \\
		0 & 0 
	\end{array} \right]  \\
	[X,[Y,[X,[Y,X]]]] &=  \left[  \begin{array}{cc}
		0 & 0 \\
		-4h^2(D+F)^3 & 0 
	\end{array} \right]  \\
\end{align*}

All this leads to the conclusion, that the leading error terms of compact splittings are bounded by 
\\

$$ Ch^5\left[\left(\widetilde{ \mathcal{L}} +\widetilde{ \mathcal{A}}  + (\widetilde{ \mathcal{L}} +\widetilde{ \mathcal{A}} )^2  \right) \Vert \varphi  \Vert_{s+2}+\left(1+ \widetilde{ \mathcal{L}} +\widetilde{ \mathcal{A}} \right) \Vert \varphi  \Vert_{s+4}+\Vert \varphi  \Vert_{s+6}  \right]  $$
\\
in the case of $\Gamma^{[4]}_1$ and by\\

$$Ch^5 \left[ \left(  \left( \widetilde{ \mathcal{L}} +\widetilde{ \mathcal{A}} \right) ^2 +  \left( \widetilde{ \mathcal{L}} +\widetilde{ \mathcal{A}} \right) ^3\right) \Vert \varphi  \Vert_{s} +\left( \widetilde{ \mathcal{L}} +\widetilde{ \mathcal{A}} +\left( \widetilde{ \mathcal{L}} +\widetilde{ \mathcal{A}} \right) ^2\right) \Vert \varphi  \Vert_{s+2}+\left( 1+\widetilde{ \mathcal{L}} +\widetilde{ \mathcal{A}} \right) \Vert \varphi  \Vert_{s+4}+\Vert \varphi  \Vert_{s+6}\right] $$
\\
in the case of $\Gamma^{[4]}_2$.

\section{The structure of the error}\label{ERRORS}

Section \ref{ESTIMATES}  was concerned with estimates of the elements of  exponentiated matrices, while seeking bounds on the error, we are interested in  differences between  exponentials of  matrices. To establish the connection between these objects let us start with the following proposition, whose proof is trivial.

\begin{proposition}\label{proposition_error}
	Let us assume that $2\times 2$ matrix $ Z $ can be expressed in the form  $ Z = X + \xi \left[  \begin{array}{cc}
		1	 & 1\\
		1& 1
	\end{array} \right] $,  where $\xi\in\mathbf{R}$ stands for the magnitude of the perturbation of matrix $Z$. Then 	 
	$$
	\ee^{Z} = \ee^{X} + 
	\mathcal{O}(\xi) \left[  \begin{array}{cc}
		1	 & 1\\
		1& 1
	\end{array} \right].
	$$
\end{proposition}

\noindent
In the following theorem we estimate the local error by its leading term, which for sufficiently small time step is a valid and consistent error estimate. By abuse of notation we will not distinguish between scalar and vectorial functions and use the corresponding point wise $H^s$ Sobolev norm.

\begin{theorem}[Error scaling]\label{thm:errorS}
	Let us assume that Assumption \ref{MainAssumption} holds and let us denote by $\phi^t$ the exact flow, i.e., $z(t) = \phi^t(z_0)$ and by $\Phi^h$ the numerical flow, i.e.,
	$$
	z^{n+1}= \Phi^h(z^n)
	$$
	where $\Phi^h$ corresponds to Algorithm $\Gamma^{[4]}_1$ or $\Gamma^{[4]}_2$. Then, for $\varphi \in H^{s+6}(\mathbb{T}^d)$ we have that
	\begin{align*}
		\Vert  \phi^h(\varphi)
		-  \Phi^h(\varphi)\Vert_s &\leq  
		C\left[\left(  \widetilde{ \mathcal{L}}  \: \widetilde{ \mathcal{A}}+\widetilde{ \mathcal{A}}^2\right)  \min \left\lbrace h^{3},\dfrac{h^2}{ \omega_{\min}},  h^5\omega_{\max}^2 \right\rbrace  + h^5\left( \widetilde{ \mathcal{L}}  + \widetilde{ \mathcal{A}}\right) ^2 +  h^5\left( \widetilde{ \mathcal{L}}  + \widetilde{ \mathcal{A}}\right) ^3\right] \Vert \varphi  \Vert_{s} \\
		\nonumber
		& + C \left[\widetilde{ \mathcal{A}} \min \left\lbrace h^{3},\dfrac{h^2}{ \omega_{\min}},  h^5\omega_{\max}^2 \right\rbrace + h^5 \left(\widetilde{ \mathcal{L}}  + \widetilde{ \mathcal{A}}+\left( \widetilde{ \mathcal{L}}  + \widetilde{ \mathcal{A}}\right) ^2 \right)  \right]  \Vert \varphi  \Vert_{s+2}\\
		\nonumber		
		& + C \left[  h^5 \left( 1+ \widetilde{ \mathcal{L}}  + \widetilde{ \mathcal{A}} \right)  \right]  \Vert \varphi  \Vert_{s+4}\\
		\nonumber
		& + C   h^5  \Vert \varphi  \Vert_{s+6}
	\end{align*}
	where constants $\widetilde{ \mathcal{A}}$, $\widetilde{ \mathcal{L}}$  are defined as in Definition \ref{errors_def}.
\end{theorem}

The proof of the theorem follows directly from the estimates in Section \ref{ESTIMATES} and from Proposition \ref{proposition_error}.

\medskip

\begin{remark} Let us discuss the error bound of Theorem \ref{thm:errorS} in more detail.
	Note that if the input function $f(t)$ does not involve a highly oscillatory term, that is $a_n\equiv0$, then all derived methods are of order 5 (locally). If, however, some of the coefficients $a_n$ are non-zero, then it is enough to take time steps $h \gtrsim 1/\sqrt[3]{ \min_n\vert \omega_n\vert }$ to obtain an order 5 (locally)  method.  This phenomenon is also observed in  numerical examples, see Section \ref{EXPERIMENTS}. Another interesting observation is that for purely oscillatory input term $f(t)$ (that is when $\alpha(t)\equiv0$), we can obtain a method of order $r$, that is of local error $\mathcal{O}(h^{r+1})$ by choosing time step $h \gtrsim 1/\sqrt[r-1]{ \min_n\vert \omega_n\vert }$.
\end{remark}

The interplay between the oscillatory parameter $\omega$ and  time step $h$  is illustrated in Figures \ref{fig:orders3}--\ref{fig:order_1D_2D}.

\begin{figure}[h]
	\centering
	\captionsetup{justification=centering,margin=0cm}
	\includegraphics[width=0.87\linewidth]{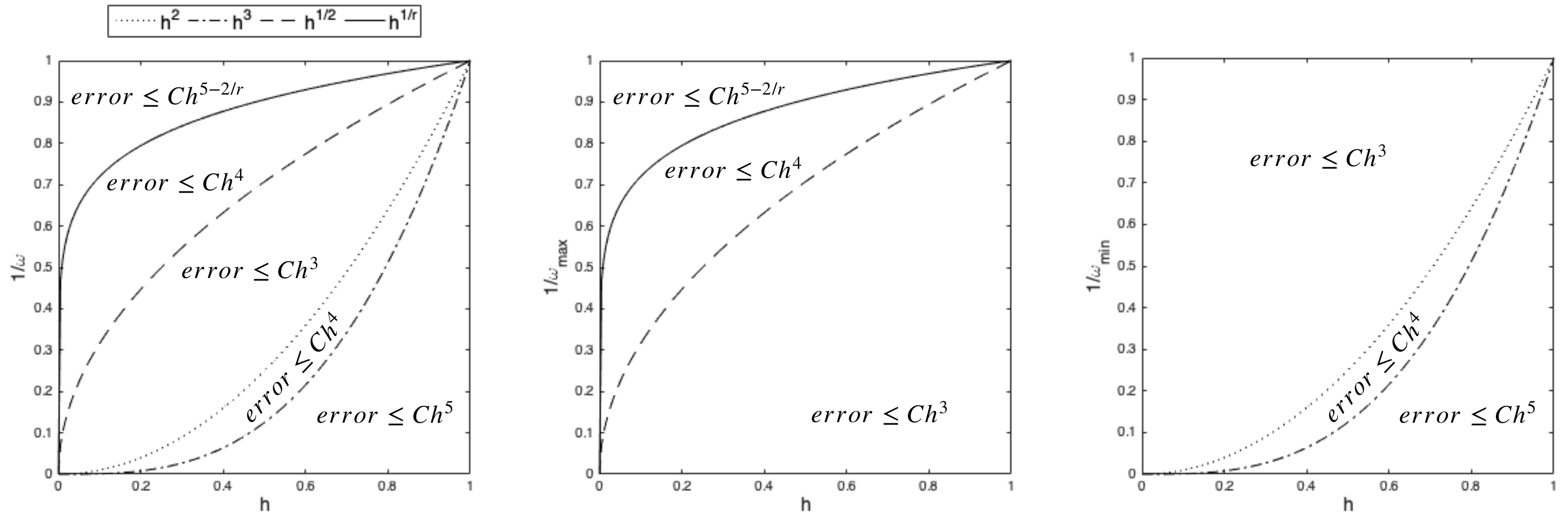}
	\caption{The local error of the proposed numerical methods strongly depends on the oscillatory nature of input term and the ratio between the oscillations $\omega_n$ and time step $h$. The left figure illustrates the accuracy obtained for the monochromatic case ($\omega_{\min}=\omega_{\max}$). The middle figure shows the dependency between the fastest frequency $\omega_{\max}$ and time step $h$, while the right figure  illustrates   the ratio between the slowest frequency $\omega_{\min}$ and time step $h$. The final accuracy (presented in Figure \ref{fig:order_1D_2D}) depends, however on both: $\omega_{\min}$ and $\omega_{\max}$.  In the above graphs we assume  without loss of generality  that  $\alpha (t)\neq 0$. }
	\label{fig:orders3}
\end{figure}

\begin{figure}[t]
	\centering
	\captionsetup{justification=centering,margin=0cm}
	\includegraphics[width=1\linewidth]{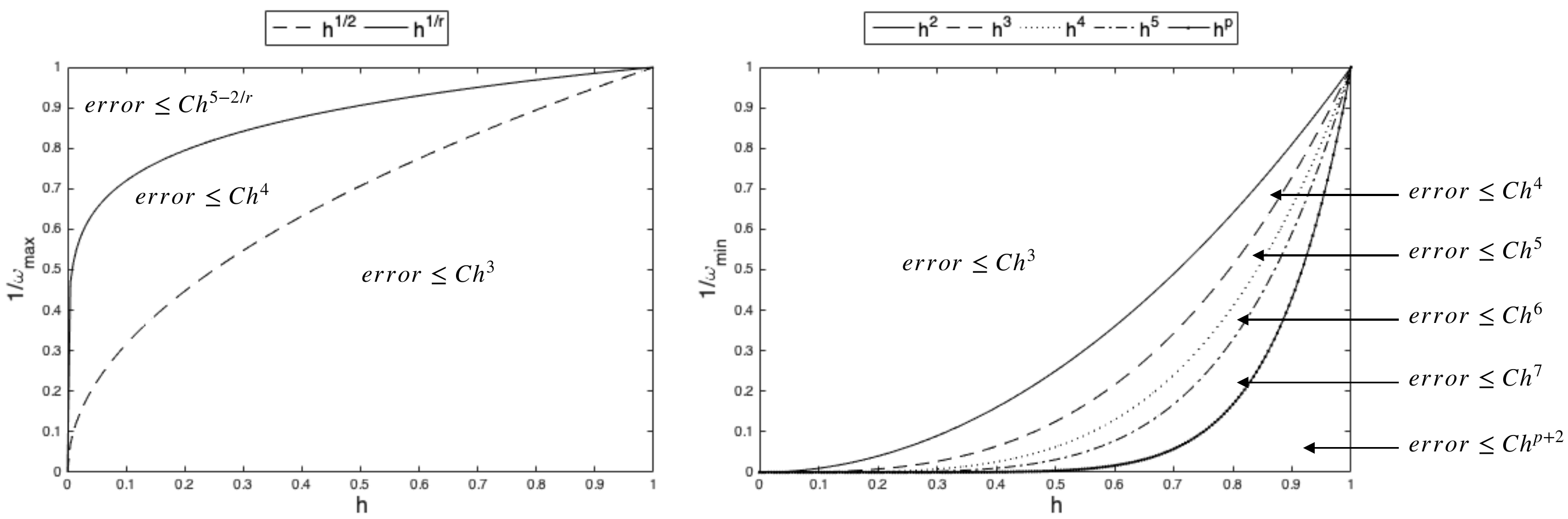}
	\caption{ The left figure illustrates the accuracy in regard of ratio between the fastest frequency $\omega_{\max}$ and time step $h$, while the right figure  illustrates  the ratio between the slowest frequency $\omega_{\min}$ and time step $h$.   Unlikely in Figure \ref{fig:orders3} we assume here that $\alpha (t)= 0$.}
	\label{fig:orders_oscil}
\end{figure}

\newpage
\begin{figure}[h]
	\centering
	\captionsetup{justification=centering,margin=0cm}
	\includegraphics[width=1\linewidth]{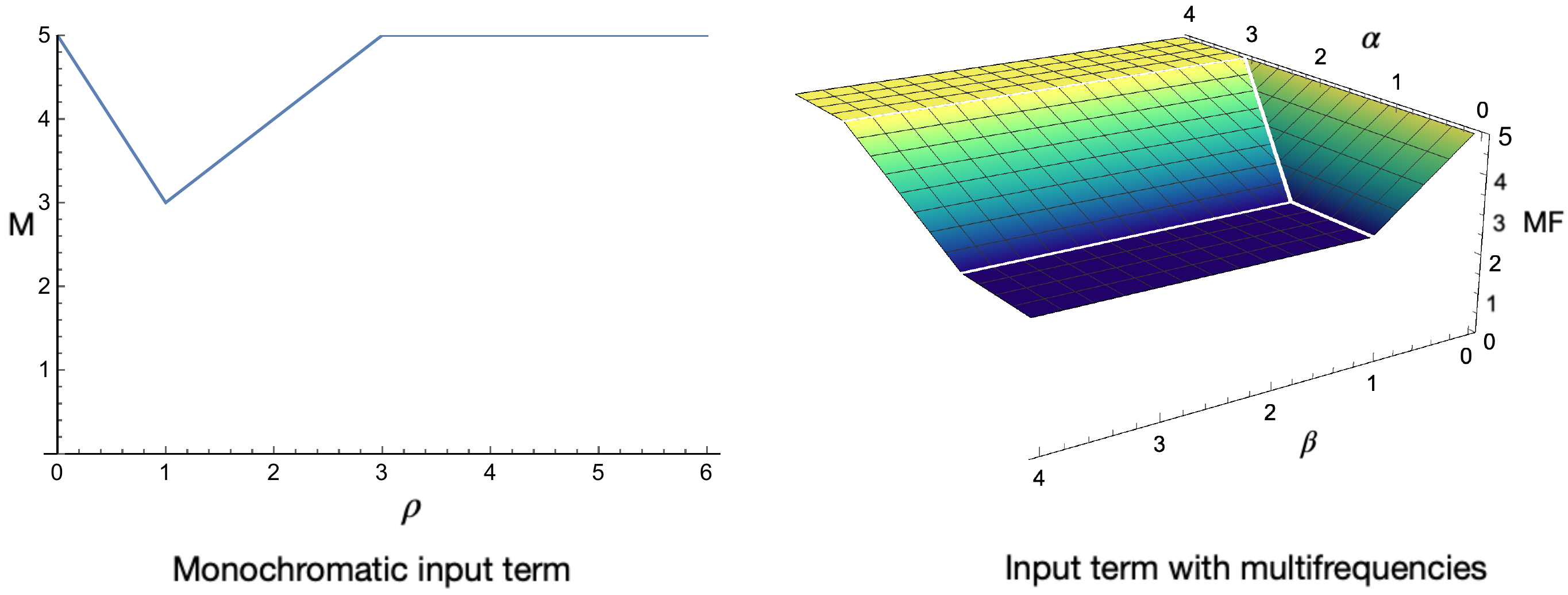}
	\caption{The left figure illustrates the accuracy obtained in the monochromatic case ($\omega_{\min}=\omega_{max}$) and should be understood in the following way: for $\omega_{\min}=h^{-\rho}$ the error of approximation scales like $\mathcal{O}(h^M)$, where $M=\max\{3,2+\rho,5-2\rho\}$. The figure on the right hand side covers the case of multiple frequencies, where for $\omega_{\min}=h^{-\alpha}$, $\omega_{\max}=h^{-\beta}$ the accuracy  of the method is $\mathcal{O}(h^{MF})$, and $MF=\max\{3,2+\alpha,5-2\beta \}$.}
	\label{fig:order_1D_2D}
\end{figure}

\section{Numerical experiments}\label{EXPERIMENTS}

In this section we compare our newly constructed numerical methods with several schemes from the literature. In this respect, the following  methods are considered: 
\begin{itemize}
	\item \textbf{BBCK} $\mathbf{^{[4]}}  $: 4th order method $ \Sigma_{3c}^{[4] }$ from \cite{baderblanes};
	\item  \textbf{BBCK}$\mathbf{^{[6]} } $:  6th order method $ \Sigma_{5c}^{[6] }  $ from \cite{baderblanes};
	\item \textbf{Asympt}$\mathbf{^{[3]}} $: 3rd order asymptotic method from \cite{asympt};
	\item $ \mathbf{\Xi^{[3]} }$: 3rd order method from \cite{3order}.
\end{itemize}

For our experiments we use Fourier method described in \cite{kopriva, trefethen2000spectral}, using 200 modes. As  a reference solution we take  optimal 6-th order Magnus integrator \cite{blanesros}  with a  step size $ h = 10^{-6} $, subsequently carry out the numerical integration with each method using different time steps and measure the $ \ell_2 $ error at the final time. This error is plotted in double-logarithmic scale versus the order of the method and calculation time expressed in seconds.

\medskip
\noindent
\textit{Example 1.} \\
\noindent
In first example  we take a wave equation with time-dependent potential and with frequency $ \omega $, namely 
\begin{eqnarray}\label{example_3}
	\partial_t^2 u &=&  \partial_x^2 u - \left(  1+\varepsilon \cos \omega t \right)\! x^2 u, \quad x\in [-10,10] , \quad t \in [0, 1], \\
	\nonumber
	u(x,0) &=& \ee^{- x^2/2} , \qquad \partial_t u (x,0) = 0, \\
	\nonumber
	u(-10,t) &=& u(10,t) , \quad t \in [0, 1] .
\end{eqnarray}
On the first graph in Figure \ref{fig:4}  we present the initial condition (blue line),  solution at final time step  for $ \varepsilon  = 10, \omega  = 10 $ (yellow line) and solution at final time step  for  $ \varepsilon  = 0.1, \omega  =100 $ (red line). Next two graphs show the evolutions in time of the solutions for $ \varepsilon = 0.1, \omega  = 100 $ and for $\varepsilon = 10, \omega  =10 $.  

\begin{figure}[h]
	\centering
	\includegraphics[width=0.8\linewidth]{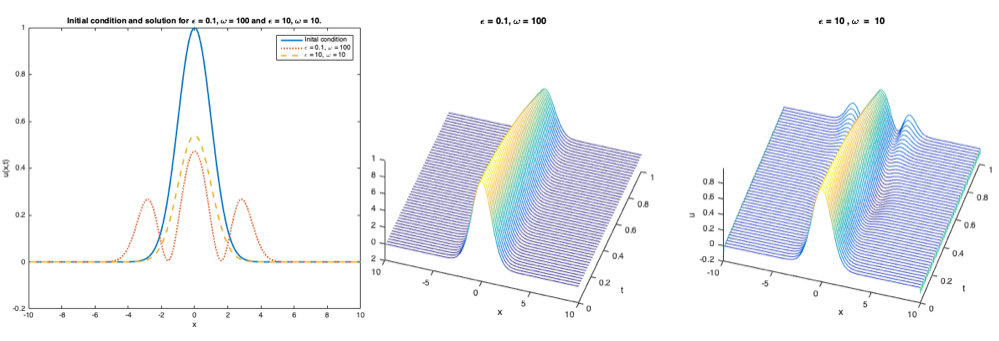}
	\caption{Initial condition, solutions at final time step  and evolution of the solutions in time of (\ref{example_3}) for two pairs of coefficients $ \varepsilon  $ and $ \omega. $}
	\label{fig:4}
\end{figure}

Comparisons of cost {\em vs\/} accuracy for equation (\ref{example_3}) are presented  in Figures \ref{fig:5} and \ref{fig:8}. First of all we observe that the asymptotic method Asympt$^{[3]} $ proposed  in \cite{asympt} is unbeatable for an equation  with extremely oscillatory inhomogeneous terms. Indeed, that asymptotic method was designed with this type of equations in mind. For small $ \omega $ asymptotic method Asympt$^{[3]} $ is completely ineffective, though.

\begin{figure}[h]
	\centering
	\includegraphics[width=0.8\linewidth]{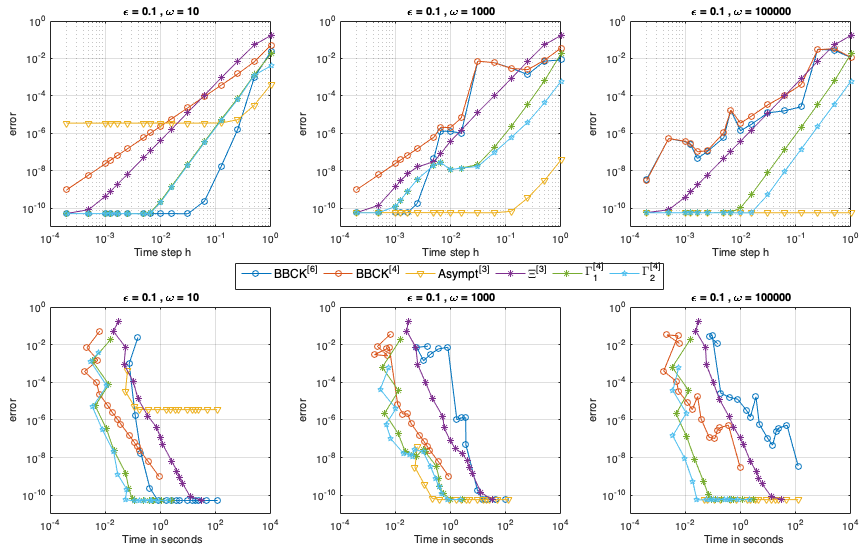}
	\caption{Comparison of order (top row) and of runtime in seconds (bottom row) for equation (\ref{example_3}) for  $  \varepsilon  = 0.1$. }
	\label{fig:5}
\end{figure}
\newpage

In the case of smaller oscillations, when $\omega=10,$  methods $\Gamma^{[4]}_1$ and $\Gamma^{[4]}_2$ preform predictably -- they achieve worse error than the 6th order  BBCK$^{[6]}  $ method, but deliver significantly smaller error  than 4th order method BBCK$^{[4]} $ and  the asymptotic method Asympt$^{[3]} $. As the oscillations get larger, $ \omega = 1000 $, asymptotic method Asympt$^{[3]} $ starts behaving extraordinary - as it was designed - especially and only for extremely large oscillations. Methods BBCK$^{[4]} $ and BBCK$^{[6]} $ require exceedingly small time  steps to handle the oscillations, while our new methods deliver expected errors for all time steps $h$. Moreover, for $h>10^{-2}$ method BBCK$^{[6]}  $ coincides with the method BBCK$^{[4]}.  $  In case of extremely high oscillations, $ \omega = 10^{6}, $ methods BBCK$^{[4]} $ and BBCK$^{[6]} $ present the same order of convergence, close to the second order, while the methods $\Gamma^{[4]}_1$ and $\Gamma_2^{[4]}$ achieve  4th order of convergence from the first, largest, time step $ h = 1 $,  consistently with their theory. Computational cost of methods $\Gamma^{[4]}_1$ and $\Gamma^{[4]}_2$ is low in case of all frequencies $ \omega. $ Method $ \Xi^{[3]} $ preserves  third order regardless of the size of the oscillations $  \omega $. However, we observe that the accuracy of methods $\Gamma^{[4]}_1$ and $\Gamma_2^{[4]}$ is much better. 

\begin{figure}[h]
	\centering
	\includegraphics[width=0.8\linewidth]{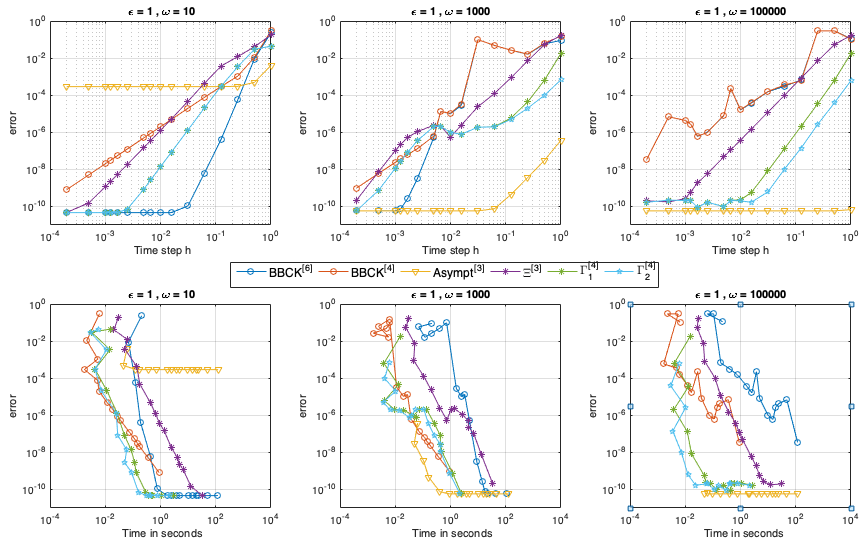}
	\caption{Comparison of orders (top row) and of costs in seconds (bottom row) for equation (\ref{example_3}) and  $  \varepsilon  = 1$. }
	\label{fig:8}
\end{figure}

In Figure \ref{fig:10} we present the comparison of error constants of methods $\Gamma^{[4]}_1$ and $\Gamma_2^{[4]}$ for various chooses of size of spatial step $ \Delta_x. $ Size of a spatial step does not effect the accuracy of the approximation. 

\begin{figure}[h]
	\centering
	\includegraphics[width=0.8\linewidth]{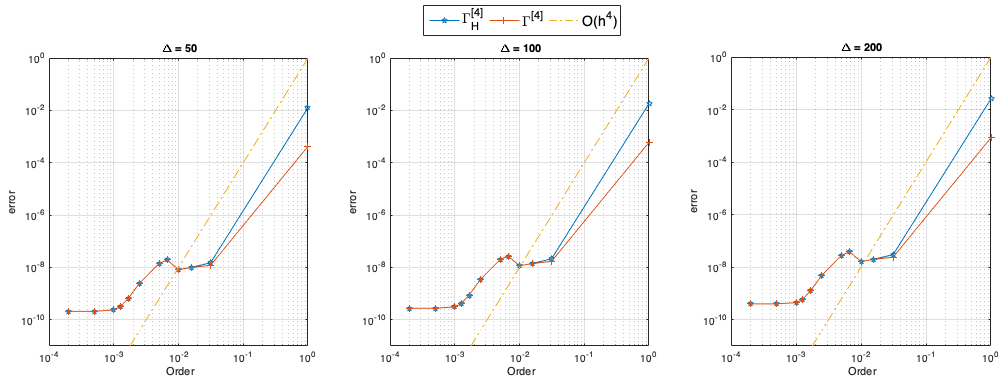}
	\caption{Comparison of error constant of methods $\Gamma^{[4]}_1$ and $\Gamma_2^{[4]}$ for equation (\ref{example_3}) for  $  \varepsilon  = 0.1$ and $ \omega = 1000$.   }
	\label{fig:10}
\end{figure}

\newpage

\medskip
\noindent
\textit{Example 2.}\\
\noindent
Let us consider example featuring large disproportion between the laplacian part and the input term part, where the laplacian part is multiplied by the factor $ 10^{-3} $, and the input term is multiplied by its inverse, i.e. $ 10^3. $
\begin{eqnarray}\label{example_2}
	\partial_t^2 u &=& 10^{-3} \partial_x^2 u - 10^3 \left( 1+ \frac{1}{5} \cos \omega t\right)\! x^2 u, \quad x \in [-\pi,\pi] , \quad t \in [0, 1] ; \\
	\nonumber
	u(x,0)  &=&  \ee^{-\frac{1}{2} (x-3)^2} +  \ee^{-\frac{1}{2} (x+3)^2} , \quad u'(x,0) = 0 ; \\
	\nonumber
	u(-\pi,t)  &=&  u(\pi,t) , \quad t \in [0, 1] .
\end{eqnarray}

In Figure \ref{fig:3} we compare the order of methods (in top row) and cost in seconds (bottom row). It is  easy to see that the lack of proportion between the Laplacian part and the input term part, which can be understood as semiclassical-like regime, has a negative effect on all considered methods. For $\omega=1$  the method $\Gamma^{[4]}_1$ is less computationally costly and obtains better accuracy than $\Gamma^{[4]}_2$. Moreover, in this case the method $\Gamma^{[4]}_1$  achieves an error only slightly larger than the 6th order method BBCK$^{[6]}  $. However, once the oscillation increases, the $\Gamma^{[4]}_1$ and $\Gamma^{[4]}_2$ methods require considerably smaller time steps than the BBCK$^{[6]}$ method, but larger than the BBCK$^{[4]}$ method. Indeed,  for  $ \omega = 500 $ and time step $ h = 10^{-4} $ the 6th order method BBCK$^{[6]}  $ obtains error of size $  10^{-11}$, while 4th order method BBCK$^{[4]}$ achieves error of size $10^{-4}$  and the new 4th order methods $\Gamma^{[4]}_1$  and $\Gamma^{[4]}_2$ obtain error of size $  10^{-8}$.  Having said this, methods $\Gamma_1^{[4]}$ and $\Gamma_2^{[4]}$ are less  computationally costly then methods BBCK$^{[4]}  $ and BBCK$^{[6]}  $. 

\medskip

\begin{figure}[h]
	\centering
	\includegraphics[width=0.9\linewidth]{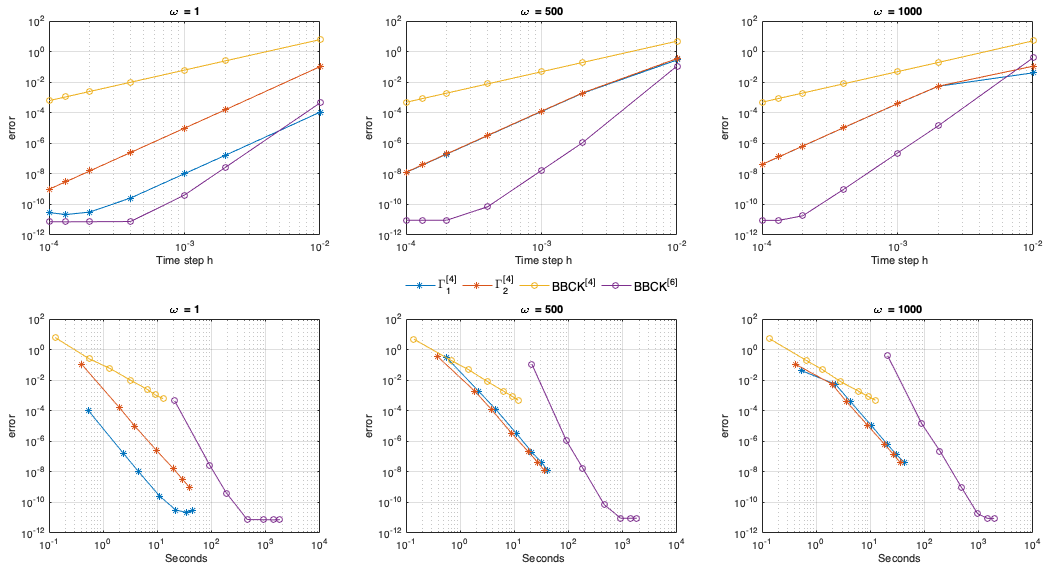}
	\caption{Comparison of orders (first row) and of costs in seconds (second row) for equation (\ref{example_2}). }
	\label{fig:3}
\end{figure}

\medskip 
\noindent
\textit{Example 3.} \\
\noindent
In this example we consider a wave equation in two spatial dimensions, 
\begin{eqnarray}\label{example_4}
	\partial_t^2 u  &=&  \Delta  u - \sigma \left( 1+ \frac{1}{5} \cos \omega t\right) \!x^2 y^2u, \quad x,y \in [-\pi,\pi] , \quad t \in [0, 1] , \\
	\nonumber
	u(x,y,0)  &=&  \ee^{-\frac{1}{2} (x-3)^2} +  \ee^{-\frac{1}{2} (x+3)^2}+ \ee^{-\frac{1}{2} (y-3)^2} +  \ee^{-\frac{1}{2} (y+3)^2} , \quad u'(x,y,0) = 0;\\
	\nonumber
	u(-\pi,-\pi,t)  &=&  u(\pi,\pi,t) = u(\pi,-\pi,t) = u(\pi,-\pi,t), \quad t \in [0, 1].
\end{eqnarray}

The above example is used to show that  the methods $\Gamma^{[4]}_1$ and $\Gamma_2^{[4]}$  can also be used for problems in higher spatial dimensions. Although we present an example in two spatial dimensions, these methods can be  extended to higher dimensions, but as the number of spatial dimensions $ d $ increases, the size of the matrix resulting from semidiscterization grows as $ M^d $. Thus, such calculations require either significant outlay in computing power or further range of tools from numerical linear algebra to reduce computational cost.  Accuracy and efficiency are presented in Figure \ref{fig:6}.  

\begin{figure}[h]
	\centering
	\includegraphics[width=0.9\linewidth]{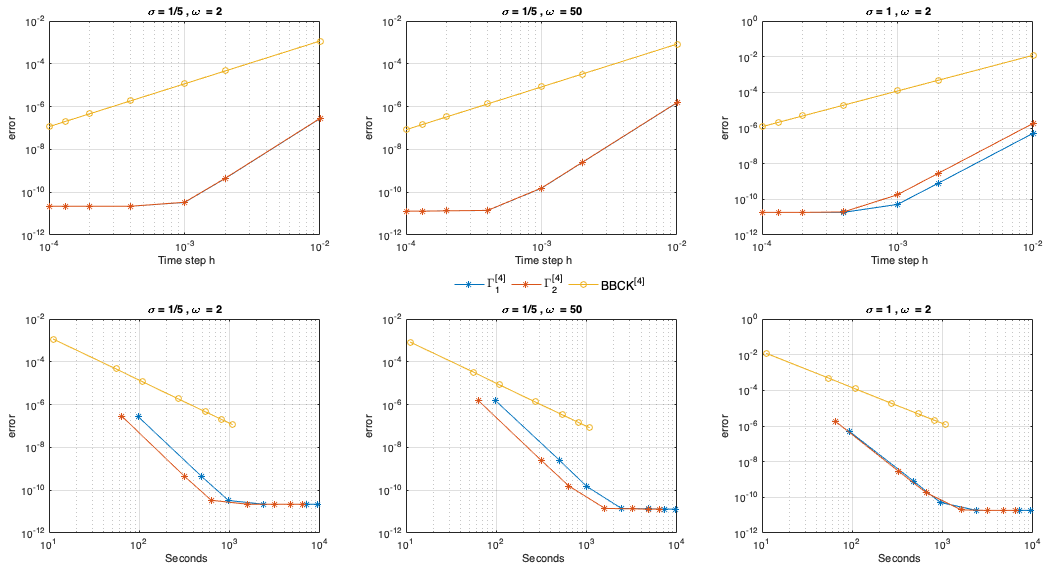}
	\caption{Comparison of orders (first row) and of costs in seconds (second row) for equation (\ref{example_4})  (where $\xx\in \R^2$) . 
	}
	\label{fig:6}
\end{figure}

\medskip
\noindent
\textit{Example 4.} \\
\noindent
In this example we consider a wave equation 
\begin{eqnarray}\label{example_5}
	\partial_t^2 u &=&  \partial_x^2 u - \dfrac{1}{1+t^2}\left( 1+ \frac{1}{5} \cos \omega t\right)\! x^2 u, \quad x \in [-\pi,\pi], \quad t \in [0, 1]; \\
	\nonumber
	u(x,0) &=&  \ee^{-\frac{1}{2} (x-3)^2} +  \ee^{-\frac{1}{2} (x+3)^2} , \quad u'(x,0) = 0 ; \\
	\nonumber
	u(-\pi,t) &=&  u(\pi,t) , \quad t \in [0, 1] ,
\end{eqnarray}
where the non-oscillatory part is $  -\dfrac{1}{1+t^2} x^2 $ and the oscillatory part $ \dfrac{1}{5} \dfrac{1}{1+t^2}  \cos (\omega t) x^2   $ are both time and space dependent. Accuracy of the methods is displayed in Figure \ref{fig:7}.  Methods $ \Gamma^{[4]}_1$ and $\Gamma^{[4]}_2  $ present the same accuracy and error constant.  $ \Gamma^{[4]}_1$ is slightly more computationally costly then $\Gamma^{[4]}_2  $, but both methods are considerably more accurate then method BBCK$^{[4]}  $  and much less computationally costly then  the 6th order method BBCK$^{[6]}. $ Method $ \Xi^{[3]} $ preserves  third order of the convergence uniformly in $ \omega $, while  methods $\Gamma^{[4]}_1$ and $\Gamma_2^{[4]}$ are of order four with much smaller error constant than method $ \Xi^{[3]} $.

\begin{figure}[h]
	\centering
	\includegraphics[width=0.9\linewidth]{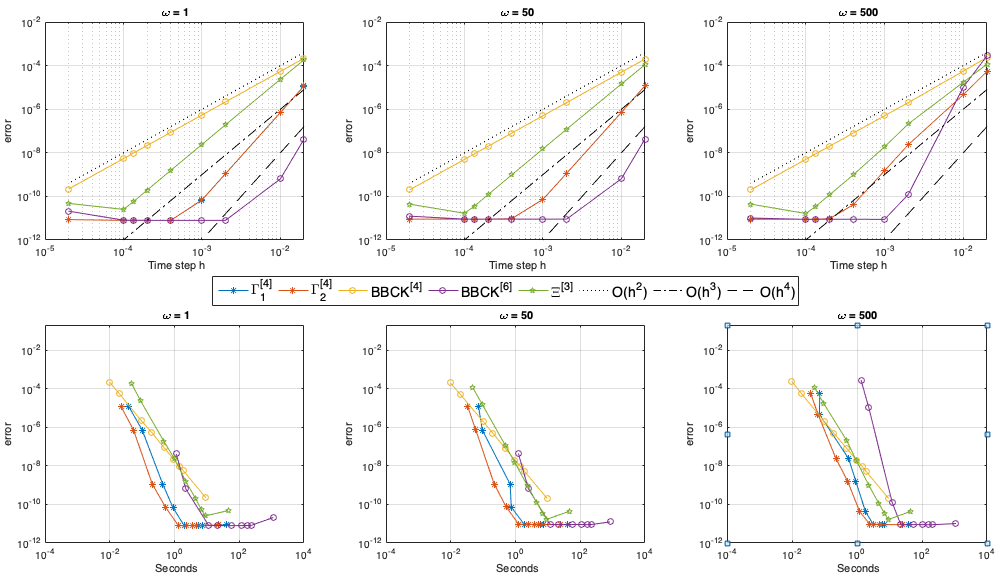}
	\caption{Comparison of orders (first row) and of cost in seconds (second row) for equation (\ref{example_5}).}
	\label{fig:7}
\end{figure}

\newpage
\medskip
\noindent\textit{Example 5.}\\
\noindent
Last but not least we consider an equation where the input term features a wide range of frequencies,
\begin{eqnarray}\label{example_6}
	\partial_t^2 u &=&  \partial_x^2 u - \sum_{k=0}^{5} \left(  1+\cos (10^k t) \right) x^2 u, \quad x\in [-\pi,\pi] , \quad t \in [0, 1], \\
	\nonumber
	u(x,0) &=& e^{-\frac{1}{2} (x-3)^2} +  e^{-\frac{1}{2} (x+3)^2} , \quad \partial_t u (x,0) = 0, \\
	\nonumber
	u(-\pi,t) &=& u(\pi,t) , \quad t \in [0, 1] .
\end{eqnarray}
On this example we can see, that order of the method is not always  crucial - method $ \Xi^{[3]} $ preserves  third order uniformly,  while  methods $\Gamma^{[4]}_1$ and $\Gamma_2^{[4]}$ have much better accuracy. Moreover,  methods $\Gamma^{[4]}_1$ and $\Gamma_2^{[4]}$ are significantly less computationally costly  than other presented methods. 
\begin{figure}[h]
	\centering
	\includegraphics[width=0.9\linewidth]{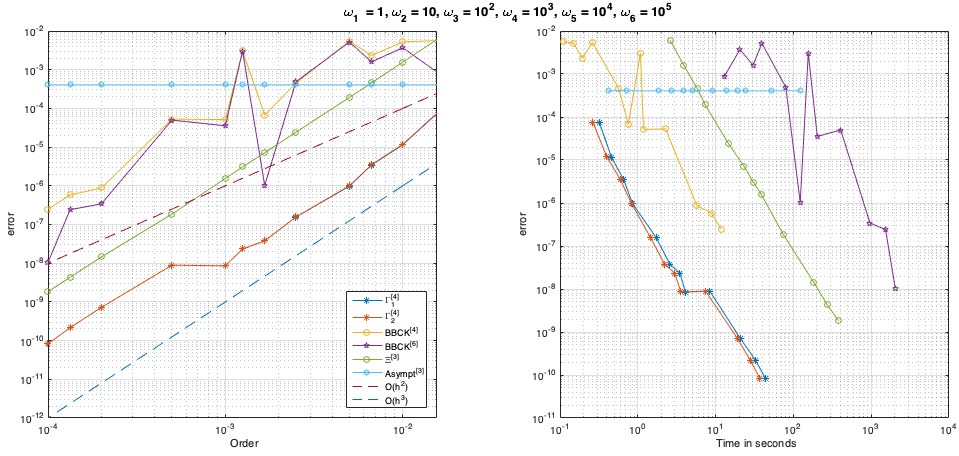}
	\caption{ We present the results for equation (\ref{example_6}) where the input term features very slow oscillations  $\omega_1=1$ and high,  $\omega_6=10^5$. We can observe that the new methods did not attain fourth order convergence (because the time step $h=10^{-4}$ is larger than $\omega_{\rm max}^{-1}=10^{-5}$) yet the accuracy is significantly better than of all other methods. Moreover, computational times of methods $\Gamma^{[4]}_1$ and $\Gamma_2^{[4]}$ are  much smaller compared to other methods. 
	}
	\label{fig:omegas_sum}
\end{figure}

According to  the error estimates (see Figures \ref{fig:orders3}--\ref{fig:order_1D_2D}  and Theorem \ref{thm:errorS}) the new methods commit an error scaling at least like $\mathcal{O}(h^2)$ globally. In Figure \ref{fig:omegas_sum} we can observe, that the performance of new methods is fairly steadily for all time steps. Moreover,  as expected,  the asymptotic method performs very purely because of the presence of the low frequency $\omega_1$. On the other hand, large oscillations like $\omega_6$ sabotage other Magnus-based methods because of their approximation of highly oscillatory  integrals.

\subsection*{Acknowledgments}

We are grateful to Arieh Iserles for his friendship, encouragement and invaluable editorial support.

The work of Katharina Schratz  in this project was funded by the European Research Council (ERC) under the European Union's Horizon 2020 research and innovation programme (grant agreement No. 850941).

The work of Karolina Kropielnicka and Karolina Lademann in this project was funded by the National Science Centre (NCN) project no. 2019/34/E/ST1/00390.

Numerical simulations were carried out by Karolina Lademann at the Academic Computer Center in Gda\'{n}sk (CI TASK).

The authors with to thank the Isaac Newton Institute for Mathematical Sciences for support and hospitality during the programme ``Geometry, compatibility and structure preservation in computational differential equations'', supported by EPSRC grant EP/R014604/1, where this work has been initiated.

This work was partially financed by Simons Foundation Award No. 663281 granted to the Institute of Mathematics of the Polish Academy of Sciences for the years 2021--2023.

\appendix
\def\theequation{\Alph{section}.\arabic{equation}}

\section{Simplification in equations (\ref{mag1})--(\ref{mag4})}\label{apA}

Let us recall that $ f(t) =  \alpha(t) + \sum_{|n|\leq N} a_n(t) \ee^{i \omega_n t} $ and
$$ A( t_1)  = \left[ \begin{array}{cc}
	0 & 1\\
	\Delta  +  f(t_1) & 0
\end{array} 
\right], \quad 
[A(t_2),A(t_1)] = 
\left[  
\begin{array}{cc}
	f( t_1)-f(t_2) & 0 \\
	0 & f( t_2)-f(t_1)
\end{array}
\right].$$

In the following part we are calculating twofold and threefold nested commutators. 
\begin{eqnarray*}
	\left[ A( t_1) , \left[ A( t_2) , A( t_3) \right] \right] & =& \left[  \begin{array}{cc}
		0 & 1 \\
		\Delta +f(t_1) & 0
	\end{array} \right] \left[  \begin{array}{cc}
		f(t_3)-f( t_2) & 0 \\
		0 &f( t_2)-f(  t_3)
	\end{array} \right] \\
	&&\mbox{}- \left[  \begin{array}{cc}
		f(t_3)-f(t_2) & 0 \\
		0 &f( t_2)-f( t_3)
	\end{array} \right]  \left[  \begin{array}{cc}
		0 & 1 \\
		\Delta +f(  t_1) & 0
	\end{array} \right] \\
	&=& \left[  \begin{array}{cc}
		0 & f( t_2)-f( t_3) \\
		\left[\Delta +f( t_1)\right] \left[f(  t_3)-f(  t_2) \right]   & 0
	\end{array} \right] \\
	&&\mbox{}- \left[  \begin{array}{cc}
		0 & f(  t_3)-f( t_2) \\
		\left[f( t_2)-f( t_3) \right]\left[\Delta +f( t_1)\right]    & 0
	\end{array} \right] \\
	&=& \left[  \begin{array}{cc}
		0 & H_1 \\
		H_2  & 0
	\end{array} \right]\!,
\end{eqnarray*}
where
\begin{eqnarray*}
	H_1 & = & 2 \left[f(  t_2)-f(t_3) \right]\\
	\nonumber
	H_2 & =& \left[\Delta +f( t_1)\right] \left[f( t_3)-f(  t_2) \right]  - \left[f( t_2)-f(  t_3) \right]\left[\Delta +f(t_1)\right]  \\
	& =&  \Delta  f( t_3) + f(  t_3) \Delta  - \Delta  f( t_2) - f( t_2) \Delta  + 2 f(  t_1) f(  t_3) - 2 f(  t_1) f( t_2) . 
\end{eqnarray*}
Analogously
\begin{displaymath}
	\left[ A( t_3) , \left[ A( t_2) , A( t_1) \right] \right] =  \left[  \begin{array}{cc}
		0 & H_3\\
		H_4 & 0
	\end{array} \right]\!,
\end{displaymath}
where
\begin{eqnarray*}
	H_3& = & 2 \left[f( t_2)-f(  t_1) \right],\\
	H_4 & =& \Delta  f( t_1) + f(  t_1) \Delta  - \Delta  f(  t_2) - f(  t_2) \Delta  + 2 f(  t_3) f(t_1) - 2 f( t_3) f(  t_2).
\end{eqnarray*}

For the threefold nested commutators we have 
\begin{eqnarray*}
	\left[A(t_4),  \left[ \left[ A(t_1),A(t_2) \right] , A(t_3)\right] \right]& = &  \left[  \begin{array}{cc}
		0 & 1 \\
		\Delta +f(  t_4) & 0
	\end{array} \right]\left[  \begin{array}{cc}
		0 & H_5 \\
		H_6 & 0
	\end{array} \right] \\
	&&\mbox{}-\left[  \begin{array}{cc}
		0 & H_5 \\
		H_6 & 0
	\end{array} \right] \left[  \begin{array}{cc}
		0 & 1 \\
		\Delta +f(  t_4) & 0
	\end{array} \right] \\
	&=&    \left[  \begin{array}{cc}
		H_6 &  0\\
		0 &\left[ \Delta +f(  t_4) \right] H_5
	\end{array} \right] - \left[  \begin{array}{cc}
		H_5  \left[ \Delta +f(  t_4) \right] &  0\\
		0 & H_6
	\end{array} \right]  \\
	&=& \left[  \begin{array}{cc}
		H_6 -	H_5\left[ \Delta +f(  t_4) \right]   &  0\\
		0 & \left[ \Delta +f(  t_4) \right]H_5  -H_6
	\end{array} \right]\!,
\end{eqnarray*}
where
\begin{eqnarray*}
	H_5& = & 2 \left[f(  t_2)-f(  t_1) \right],\\
	H_6 & =&\Delta  \left[ f(  t_1)  - f(  t_2)\right] +\left[ f(  t_1)  - f(  t_2)\right]  \Delta  + 2 f(  t_3) \left( f(  t_1)  - f(  t_2)\right).
\end{eqnarray*}

Likewise,
\begin{displaymath}
	\left[ A(t_1), \left[  \left[  A(t_2), A(t_3)\right] , A(t_4) \right] \right]   =
	\left[  \begin{array}{cc}
		H_8 -H_7  \left( \Delta +f(  t_1) \right) &  0\\
		0 &   \left[ \Delta +f(  t_1) \right] H_7 - H_8
	\end{array} \right]\!,
\end{displaymath}  
with
\begin{eqnarray*}
	\nonumber
	H_7& = & 2 \left[f(  t_3)-f(  t_2) \right]\\
	\nonumber
	H_8 & =&\Delta  \left[  f(  t_2) -  f(  t_3) \right]+\left[  f(  t_2) -  f(  t_3) \right]\Delta  +2 f(  t_4) \left[  f(  t_2) -  f(  t_3) \right].
\end{eqnarray*}

Finally,
\begin{eqnarray*}
	\left[ A(t_1), \left[ A(t_2),\left[ A(t_3) , A(t_4)\right] \right] \right]  & =& \left[  \begin{array}{cc}
		0 & 1 \\
		\Delta +f(  t_1) & 0
	\end{array} \right]\left[  \begin{array}{cc}
		0 & H_9 \\
		H_{10} & 0
	\end{array} \right] \\
	&&\mbox{}-  \left[  \begin{array}{cc}
		0 & H_9 \\
		H_{10} & 0
	\end{array} \right]\left[  \begin{array}{cc}
		0 & 1 \\
		\Delta +f(  t_1) & 0
	\end{array} \right] \\
	& =&\left[  \begin{array}{cc}
		H_{10}  &  0\\
		0 & \left[ \Delta +f(  t_4) \right] H_9
	\end{array} \right] - \left[  \begin{array}{cc}
		H_9 \left[ \Delta +f(  t_4) \right] &  0\\
		0 & H_{10}
	\end{array} \right] \\
	& =&\left[  \begin{array}{cc}
		H_{10} -H_9 \left[ \Delta +f(  t_1) \right] &  0\\
		0 &   \left[ \Delta +f(  t_1) \right] H_9 - H_{10}
	\end{array} \right]\!,
\end{eqnarray*}  
where
\begin{eqnarray}
	\nonumber
	H_9& = & 2 \left[f(  t_3)-f(  t_4) \right],\\
	\nonumber
	H_{10} & =& \Delta  \left[  f(  t_4) -  f(  t_3) \right]+\left[  f(  t_4) -  f(  t_3) \right]\Delta  +2 f(  t_2) \left[  f(  t_4) -  f(  t_3) \right]
\end{eqnarray}
and	
\begin{displaymath}
	\left[ A(t_2), \left[ A(t_3),\left[ A(t_4) , A(t_1)\right] \right] \right]  = 
	\left[  \begin{array}{cc}
		H_{12} -  H_{11} \left[ \Delta +f(  t_2) \right]  &  0\\
		0 &  \left[ \Delta +f(  t_2) \right]   H_{11} - H_{12}
	\end{array} \right]\!,
\end{displaymath} 	
where
\begin{eqnarray*}
	H_{11}& = & 2 \left[f(  t_4)-f(  t_1) \right]\\
	H_{12} & =& \Delta \left[  f(  t_1) -  f(  t_4) \right]+\left[  f(  t_1) -  f(  t_4) \right]\Delta  +2 f(  t_3) \left[  f(  t_1) -  f(  t_4) \right].
\end{eqnarray*}

To estimate term $ \Theta_4 $ we need to aggregate the matrices originating in individual commutators. The result is a matrix $ \left[  \begin{array}{cc}
	\mathcal{H}_1 & 0\\
	0 & \mathcal{H}_2 
\end{array} \right]\! , $ where
\begin{eqnarray*}
	\mathcal{H}_1	& =& -H_6 +	H_5\left[ \Delta +f(  t_4) \right]   + H_8-	H_7\left[ \Delta +f(  t_1) \right]  +H_{10 }-	H_9\left[ \Delta +f(  t_4) \right]  +H_{12} \\
	&&\mbox{}-	H_{11}\left[ \Delta +f(  t_2) \right]  \\
	& = & \Delta  \left[ f(  t_2)  - f(  t_1)\right] +\left[ f(  t_2)  - f(  t_1)\right]  \Delta  + 2 f(  t_3) \left[ f(  t_2)  - f(  t_1)\right] - 2 \left[f(  t_1)-f(  t_2) \right]\left[ \Delta +f(  t_4) \right] \\
	&&\mbox{} + \Delta  \left[  f(  t_2) -  f(  t_3) \right]+\left[  f(  t_2) -  f(  t_3) \right]\Delta  +2 f(  t_4) \left[  f(  t_2) -  f(  t_3) \right] - 2 \left[f(  t_3)-f(  t_2) \right]\left[ \Delta +f(  t_1) \right] \\
	&&\mbox{} +  \Delta  \left[  f(  t_4) -  f(  t_3) \right]+\left[  f(  t_4) -  f(  t_3) \right]\Delta  +2 f(  t_2) \left[  f(  t_4) -  f(  t_3) \right] - 2 \left[f(  t_3)-f(  t_4) \right]\left[ \Delta +f(  t_1) \right] \\
	& &\mbox{}+ \Delta \left([ f(  t_1) -  f(  t_4) \right]+\left[  f(  t_1) -  f(  t_4) \right]\Delta  +2 f(  t_3) \left[  f(  t_1) -  f(  t_4) \right] - 2 \left[f(  t_4)-f(  t_1) \right]\left[ \Delta +f(  t_2) \right] \\
	& =& \Delta  \left[  2f(  t_2)  -2 f(  t_3)  \right] +3 \left[  2f(  t_2)  -2 f(  t_3)  \right]\Delta  + 4f(  t_4) \left([ f(  t_2)  -f(  t_3)  \right] + 4f(  t_1) \left[  f(  t_2)  - f(  t_3)  \right]
\end{eqnarray*}
and 
\begin{displaymath}
	\mathcal{H}_2	=  -3\Delta  \left[  2f(  t_2)  -2 f(  t_3)  \right] - \left[  2f(  t_2)  -2 f(  t_3)  \right]\Delta  - 4f(  t_4) \left[  f(  t_2)  -f(  t_3)  \right] - 4f(  t_1) \left[  f(  t_2)  - f(  t_3)  \right] , 
\end{displaymath}
where 
\begin{eqnarray*}
	f(t_k) f( t_l) & =& \left[\alpha(t_k) + \sum_{|n|\leq N} a_n (t_k) \ee^{i \omega_n t_k}\right] \left[\alpha(t_l) +\sum_{|n|\leq N} a_n (t_l)  \ee^{i \omega_n t_l}\right] \\
	& =& \alpha(t_k) \alpha(t_l) + \alpha(t_k) \sum_{|n|\leq N} a_n(t_l)   \ee^{i \omega _nt_l} + \alpha(t_l) \sum_{|n|\leq N} a_n(t_k)   \ee^{i \omega_n t_k}  \\
	&&\mbox{}+ \sum_{|n|\leq N} a_n(t_k) \ee^{i \omega_n t_k} \sum_{|m|\leq N}a_m(t_l) \ee^{i \omega_m t_l} .
\end{eqnarray*}

\section{Simplification in Strang splitting error in Section \ref{STRANG_CALCULATIONS} }\label{strang_est}

Taking
\begin{eqnarray*}
	X &=& - \int_0^h\int_0^{t_1} [A(t+t_2),A(t+t_1)]\dd t_1 \dd t_2 =	\left[  \begin{array}{cc}
		-\mathcal{F}  & 0 \\
		0 & \mathcal{F} 
	\end{array} \right]\!,\\
	Y & = & \int_0^h \left[  \begin{array}{cc}
		0 & 1 \\
		\Delta  + f(t+t_1)& 0 
	\end{array} \right] \dd t_1 = \left[  \begin{array}{cc}
		0 & h \\
		h	\Delta  +	F & 0 
	\end{array} \right]\!, 
\end{eqnarray*}
we have 
\begin{eqnarray*}
	\nonumber
	[Y,X] &=&  \left[  \begin{array}{cc}
		0 & h \\
		h\Delta +F & 0 
	\end{array} \right] \left[  \begin{array}{cc}
		-\mathcal{F}  & 0 \\
		0 & \mathcal{F} 
	\end{array} \right]  - \left[  \begin{array}{cc}
		-\mathcal{F}  & 0 \\
		0 & \mathcal{F} 
	\end{array} \right] \left[  \begin{array}{cc}
		0 & h \\
		h\Delta +F & 0 
	\end{array} \right] \\
	\nonumber
	&=&  \left[  \begin{array}{cc}
		0 & h\mathcal{F} \\
		-(h\Delta +F) \mathcal{F}& 0 
	\end{array} \right]  -  \left[  \begin{array}{cc}
		0 & -h\mathcal{F} \\
		\mathcal{F}(h\Delta +F) & 0 
	\end{array} \right] \\
	&=&  \left[  \begin{array}{cc}
		0 & 2h\mathcal{F} \\
		-h \Delta  \mathcal{F} - h \mathcal{F} \Delta  - 2\mathcal{F} F & 0 
	\end{array} \right]\!, \\
	\nonumber
	[Y,[Y,X]] &=& \left[  \begin{array}{cc}
		0 & h \\
		h\Delta +F & 0 
	\end{array} \right] \left[  \begin{array}{cc}
		0 & 2h\mathcal{F} \\
		-h \Delta  \mathcal{F} - h \mathcal{F} \Delta  - 2\mathcal{F} F & 0 
	\end{array} \right] \\
	\nonumber
	&&\mbox{} -  \left[  \begin{array}{cc}
		0 & 2h\mathcal{F} \\
		-h \Delta  \mathcal{F} - h \mathcal{F} \Delta  - 2\mathcal{F} F & 0 
	\end{array} \right] \left[  \begin{array}{cc}
		0 & h \\
		h\Delta +F & 0 
	\end{array} \right] \\
	\nonumber
	&=& \left[  \begin{array}{cc}
		-h^2 \Delta  \mathcal{F} -3 h^2 \mathcal{F} \Delta   -  4h \mathcal{F} F	 & 0\\
		0	& 3h^2 \Delta  \mathcal{F} +h^2 \mathcal{F} \Delta   +  4h \mathcal{F} F	
	\end{array} \right] \! \\ \nonumber
	[X,Y] = - [Y,X]  &=&  \left[  \begin{array}{cc}
		0 & -2h\mathcal{F} \\
		h \Delta  \mathcal{F}  +h \mathcal{F} \Delta  +2\mathcal{F} F & 0 
	\end{array} \right]\,  \end{eqnarray*}

\begin{eqnarray*}
	[X,[X,Y]] &=& \left[  \begin{array}{cc}
		-\mathcal{F}  & 0 \\
		0 & \mathcal{F} 
	\end{array} \right] \left[  \begin{array}{cc}
		0 & -2h\mathcal{F} \\
		h \Delta  \mathcal{F}  +h \mathcal{F} \Delta  +2\mathcal{F} F & 0 
	\end{array} \right] \\
	\nonumber
	&&\mbox{} -  \left[  \begin{array}{cc}
		0 & -2h\mathcal{F} \\
		h \Delta  \mathcal{F}  +h \mathcal{F} \Delta  +2\mathcal{F} F & 0 
	\end{array} \right] \left[  \begin{array}{cc}
		-\mathcal{F}  & 0 \\
		0 & \mathcal{F} 
	\end{array} \right] \\
	\nonumber
	& =& \left[  \begin{array}{cc}
		0 & 4h \mathcal{F}^2\\
		2h\mathcal{F} \Delta  \mathcal{F} + h\mathcal{F}^2 \Delta  + h \Delta  \mathcal{F}^2 + 4\mathcal{F}^2 F& 0 
	\end{array} \right]\!.
\end{eqnarray*}

\bibliographystyle{apalike}

\bibliography{document}

\begin{thebibliography}{}

\bibitem[Abdulle et~al., 2012]{hmm}
Abdulle, A., E, W., Engquist, B., and Vanden-Eijnden, E. (2012).
\newblock The heterogeneous multiscale method.
\newblock {\em Acta Numer.}, 21:1--87.

\bibitem[Bader et~al., 2019]{baderblanes}
Bader, P., Blanes, S., Casas, F., and Kopylov, N. (2019).
\newblock Novel symplectic integrators for the {K}lein-{G}ordon equation with
  space- and time-dependent mass.
\newblock {\em J. Comput. Appl. Math.}, 350:130--138.

\bibitem[Bao and Dong, 2012]{12}
Bao, W. and Dong, X. (2012).
\newblock Analysis and comparison of numerical methods for the {K}lein-{G}ordon
  equation in the nonrelativistic limit regime.
\newblock {\em Numer. Math.}, 120(2):189--229.

\bibitem[Blanes et~al., 2009]{blanesros}
Blanes, S., Casas, F., Oteo, J.~A., and Ros, J. (2009).
\newblock The {M}agnus expansion and some of its applications.
\newblock {\em Phys. Rep.}, 470(5-6):151--238.

\bibitem[Blanes et~al., 2002]{blanes2002high}
Blanes, S., Casas, F., and Ros, J. (2002).
\newblock High order optimized geometric integrators for linear differential
  equations.
\newblock {\em BIT Numerical Mathematics}, 42:262--284.

\bibitem[Chen and Liu, 2008]{bao2}
Chen, J.-B. and Liu, H. (2008).
\newblock Multisymplectic pseudospectral discretizations for
  {$(3+1)$}-dimensional {K}lein-{G}ordon equation.
\newblock {\em Commun. Theor. Phys. (Beijing)}, 50(5):1052--1054.

\bibitem[Chin and Chen, 2002]{chin2002gradient}
Chin, S.~A. and Chen, C. (2002).
\newblock Gradient symplectic algorithms for solving the schr{\"o}dinger
  equation with time-dependent potentials.
\newblock {\em The Journal of Chemical Physics}, 117(4):1409--1415.

\bibitem[Condon et~al., 2021]{asympt}
Condon, M., Kropielnicka, K., Lademann, K., and Perczy\'{n}ski, R. (2021).
\newblock Asymptotic numerical solver for the linear {K}lein-{G}ordon equation
  with space- and time-dependent mass.
\newblock {\em Applied Mathematics Letters}, 115:106935.

\bibitem[Dea\~{n}o et~al., 2018]{hi-ocs-int}
Dea\~{n}o, A., Huybrechs, D., and Iserles, A. (2018).
\newblock {\em Computing highly oscillatory integrals}.
\newblock Society for Industrial and Applied Mathematics (SIAM), Philadelphia,
  PA.

\bibitem[Engquist et~al., 2009]{hop}
Engquist, B., Fokas, A., Hairer, E., and Iserles, A. (2009).
\newblock {\em Highly Oscillatory Problems}.
\newblock Number 366 in London Maths Soc. Lecture Note Series. Cambridge
  University Press.

\bibitem[Faou and Schratz, 2014]{6}
Faou, E. and Schratz, K. (2014).
\newblock Asymptotic preserving schemes for the {K}lein-{G}ordon equation in
  the non-relativistic limit regime.
\newblock {\em Numer. Math.}, 126(3):441--469.

\bibitem[Iserles et~al., 2019]{iserles2019compact}
Iserles, A., Kropielnicka, K., and Singh, P. (2019).
\newblock Compact schemes for laser--matter interaction in schr{\"o}dinger
  equation based on effective splittings of magnus expansion.
\newblock {\em Computer Physics Communications}, 234:195--201.

\bibitem[Iserles et~al., 2000]{iserles00lgm}
Iserles, A., Munthe-Kaas, H.~Z., N\o{r}sett, S.~P., and Zanna, A. (2000).
\newblock Lie-group methods.
\newblock In {\em Acta numerica, 2000}, volume~9 of {\em Acta Numer.}, pages
  215--365. Cambridge Univ. Press, Cambridge.

\bibitem[Iserles et~al., 2001]{INR2001}
Iserles, A., N\o{r}sett, S.~P., and Rasmussen, A.~F. (2001).
\newblock Time symmetry and high-order {M}agnus methods.
\newblock {\em Appl. Numer. Math.}, 39(3-4):379--401.
\newblock Special issue: Themes in geometric integration.

\bibitem[Jahnke and Lubich, 2000]{jahnke2000error}
Jahnke, T. and Lubich, C. (2000).
\newblock Error bounds for exponential operator splittings.
\newblock {\em BIT Numerical Mathematics}, 40:735--744.

\bibitem[Kopriva, 2009]{kopriva}
Kopriva, D.~A. (2009).
\newblock {\em Implementing spectral methods for partial differential
  equations}.
\newblock Scientific Computation. Springer, Berlin.
\newblock Algorithms for scientists and engineers.

\bibitem[Kropielnicka and Lademann, 2022]{3order}
Kropielnicka, K. and Lademann, K. (2022).
\newblock Third order, uniform in low to high oscillatory coefficients,
  exponential integrators for klein-gordon equations.
\newblock {\em arXiv preprint arXiv:2212.13762}.

\bibitem[Magnus, 1954]{magnus1954exponential}
Magnus, W. (1954).
\newblock On the exponential solution of differential equations for a linear
  operator.
\newblock {\em Communications on pure and applied mathematics}, 7(4):649--673.

\bibitem[Mostafazadeh, 2004]{mos1}
Mostafazadeh, A. (2004).
\newblock Quantum mechanics of {K}lein-{G}ordon-type fields and quantum
  cosmology.
\newblock {\em Annals of Physics}, 309(1):1--48.

\bibitem[Shakeri and Dehghan, 2008]{shark}
Shakeri, F. and Dehghan, M. (2008).
\newblock Numerical solution of the {K}lein-{G}ordon equation via {H}e's
  variational iteration method.
\newblock {\em Nonlinear Dynam.}, 51(1-2):89--97.

\bibitem[Trefethen, 2000]{trefethen2000spectral}
Trefethen, L.~N. (2000).
\newblock {\em Spectral methods in MATLAB}.
\newblock SIAM.

\bibitem[Yusufo\u{g}lu, 2008]{yus}
Yusufo\u{g}lu, E. (2008).
\newblock The variational iteration method for studying the {K}lein-{G}ordon
  equation.
\newblock {\em Appl. Math. Lett.}, 21(7):669--674.

\bibitem[Zanna, 2001]{zanna2001fer}
Zanna, A. (2001).
\newblock The {F}er expansion and time-symmetry: A {S}trang-type approach.
\newblock {\em Applied Numerical Mathematics}, 39(3-4):435--459.

\bibitem[Znojil, 2016]{zno3}
Znojil, M. (2016).
\newblock Quantization of big bang in crypto-hermitian heisenberg picture.
\newblock {\em Springer Proc. Phys.}, 184:383--399.

\bibitem[Znojil, 2017a]{zno1}
Znojil, M. (2017a).
\newblock {K}lein-{G}ordon equation with the time- and space-dependent
  mass:{U}nitary evolution picture.
\newblock Technical report, Nuclear Physics Institute, Czech Academy of
  Sciences.

\bibitem[Znojil, 2017b]{zno2}
Znojil, M. (2017b).
\newblock Non-{H}ermitian interaction representation and its use in
  relativistic quantum mechanics.
\newblock {\em Ann. Physics}, 385:162--179.

\end{thebibliography}

\end{document}